\documentclass[a4paper,reqno,12pt]{amsart}
\usepackage[margin=1.1in]{geometry}
\usepackage{stmaryrd} 
\usepackage{amscd}
\usepackage{amsthm}
\usepackage{amssymb}
\usepackage{amsmath}
\usepackage{mathtools}
\usepackage{enumitem}
\usepackage{epsfig}
\usepackage{tikz-cd}
\usetikzlibrary{decorations.markings,shapes.geometric,matrix,arrows,positioning}
\usepackage[hidelinks]{hyperref}
\usepackage{microtype}
\def\MR#1{} 

\tikzset{anchorbase/.style={baseline={([yshift=-0.5ex]current bounding box.center)}},
  int/.style={thick},
  cross line/.style={preaction={draw=white,line width=6pt,-}},
  wall/.style={thin,double,blue},
  middlearrow/.style={postaction=decorate,decoration={markings,mark=at
    position .55 with {\arrow{stealth};}}},
  middlearrowrev/.style={postaction=decorate,decoration={markings,mark=at
    position .55 with {\arrowreversed{stealth};}}},
  ev/.style={shape=rectangle, draw}
}

\newcommand{\tcoev}{\stackrel{\longleftarrow}{\operatorname{coev}}}
\newcommand{\tev}{\stackrel{\longleftarrow}{\operatorname{ev}}}
\newcommand{\ev}{\stackrel{\longrightarrow}{\operatorname{ev}}}
\newcommand{\coev}{\stackrel{\longrightarrow}{\operatorname{coev}}}
\newcommand{\p}[1]{\ensuremath{\bar {#1}}}

\newcommand{\C}{\mathbb{C}}
\newcommand{\R}{\mathbb{R}}

\newcommand{\Z}{\mathbb{Z}}
\newcommand{\Q}{\mathbb{Q}}

\newcommand{\I}{\sqrt{-1}}
\newcommand{\unit}{\mathbb{I}}
\newcommand{\q}{v}
\newcommand{\qq}{\mathbf{q}}
\newcommand{\cat}{\mathcal{C}}
\newcommand{\FR}{\mathsf{Z}}
\newcommand{\Gr}{\mathsf{G}}
\newcommand{\SSS}{\mathsf{X}}
\newcommand{\Hom}{\textup{\text{Hom}}}
\newcommand{\End}{\textup{\text{End}}}
\newcommand{\id}{\textup{\text{id}}}
\newcommand{\qd}{\mathsf{d}}

\newcommand{\mt}{\mathsf{tr}}
\newcommand{\Ztwo}{\mathbb{Z} \slash 2 \mathbb{Z}}
\newcommand{\osp}{\mathfrak{osp}(1 \vert 2)}
\newcommand{\sltwo}{\mathfrak{sl}(2)}
\newcommand{\gloo}{\mathfrak{gl}(1 \vert 1)}
\newcommand{\UqUn}{U_q^H(\mathfrak{osp}(1 \vert 2))}
\newcommand{\Uq}{\overline{U}_q^H(\mathfrak{osp}(1 \vert 2))}

\newcommand{\Uqb}{\overline{U}_q^{H,\geq 0}(\osp)}
\newcommand{\Uqbn}{\overline{U}_q^{H,\leq 0}(\osp)}
\newcommand{\UqG}{U_{\q}(\mathfrak{osp}(1 \vert 2))}
\newcommand{\Uqsltwo}{\overline{U}_q^H(\mathfrak{sl}(2))}
\newcommand{\ptr}{\operatorname{ptr}}
\newcommand{\tr}{\operatorname{tr}}
\def\pser#1{\llbracket#1\rrbracket}
\newcommand{\Cob}{\mathsf{C}\mathsf{ob}^{\textnormal{ad}}}

\newcommand{\ZVect}{\mathsf{V}\mathsf{ect}^{\FR\textnormal{-gr}}}

\newcommand{\TQFT}{\mathcal{Z}}
\newcommand{\CS}{{\mathcal{S}}}

\newcommand{\coh}{\omega}
\newcommand{\mer}{\mathfrak{m}}
\newcommand{\D}{{\mathcal{D}}} 
\newcommand{\Co}{\mathsf{Col}}

\newcommand{\Zhat}{\widehat{Z}}
\newcommand{\tor}{\mathcal{T}}
\newcommand{\half}{\varepsilon}
\newcommand{\even}{t}
\newcommand{\colOne}{\mathbf{1}}
\newcommand{\lk}{\ell k}

\newcommand{\Spin}{\textnormal{Spin}}
\newcommand{\Spinc}{\textnormal{Spin}^{\textnormal{c}}}

\newcommand{\red}[1]{{{#1}_0}}
\newcommand{\parity}[1]{{{#1}^{\prime}}}
\newcommand{\pideal}{I^+}
\newcommand{\mideal}{I^-}
\newcommand{\pmideal}{I}
\newcommand{\pserbin}[2]{\left\llbracket\begin{matrix}
    #1\\#2
\end{matrix}\right\rrbracket}

\makeatletter
\newsavebox{\@brx}
\newcommand{\llangle}[1][]{\savebox{\@brx}{\(\m@th{#1\langle}\)}%
  \mathopen{\copy\@brx\kern-0.5\wd\@brx\usebox{\@brx}}}
\newcommand{\rrangle}[1][]{\savebox{\@brx}{\(\m@th{#1\rangle}\)}%
  \mathclose{\copy\@brx\kern-0.5\wd\@brx\usebox{\@brx}}}
\makeatother

\newtheorem{Lem}{Lemma}[section]
\newtheorem{Prop}[Lem]{Proposition}

\newtheorem{Def}[Lem]{Definition}

\theoremstyle{plain}
\newtheorem{Thm}[Lem]{Theorem}

\newtheorem{Cor}[Lem]{Corollary}

\newtheorem{Hyp}[Lem]{Hypothesis}

\newtheorem{thm}{Theorem}

\theoremstyle{definition}
\newtheorem{examplex}[Lem]{Example}
\newenvironment{Ex}
  {\pushQED{\qed}\examplex}
  {\popQED\endexamplex}

\newtheorem{Rem}[Lem]{Remark}
  
\newcommand{\epsh}[2]
         {\begin{array}{c} \hspace{-1.3mm}
        \raisebox{-4pt}{\epsfig{figure=#1,height=#2}}
        \hspace{-1.9mm}\end{array}}

\newcommand{\comm}[1]{}

\title[TFT and $\Zhat$-invariants from $\osp$]{Non-semisimple topological field theory and $\Zhat$-invariants from $\osp$}

\author[F. Costantino]{Francesco Costantino}
\address{Institut de Math\'{e}matiques de Toulouse \\ 118 route de Narbonne, F-31062 Toulouse\\ France}
\email{francesco.costantino@math.univ-toulouse.fr}

\author[M. Harper]{Matthew Harper}
\address{Department of Mathematics \\ Michigan State University\\
East Lansing, MI 48824 \\ USA}
\address{Department of Mathematics \\ University of California Riverside\\
Riverside, CA 92521 \\ USA}
\email{mrhmath@proton.me}

\author[A. Robertson]{Adam Robertson}
\address{Department of Mathematics and Statistics \\ Utah State University\\
Logan, Utah 84322 \\ USA}
\email{adam.robertson@usu.edu}

\author[M.\,B. Young]{Matthew B. Young}
\address{Department of Mathematics and Statistics \\ Utah State University\\
Logan, Utah 84322 \\ USA}
\email{matthew.young@usu.edu}

\date{\today}

\begin{document}

\begin{abstract}
We construct three-dimensional non-semisimple topological field theories from the unrolled quantum group of the Lie superalgebra $\mathfrak{osp}(1 \vert 2)$. More precisely, the quantum group depends on a root of unity $q=e^{\frac{2 \pi \sqrt{-1}}{r}}$, where $r$ is a positive integer greater than $2$, and the construction applies when $r$ is not congruent to $4$ modulo $8$. The algebraic result which underlies the construction is the existence of a relative modular structure on the non-finite, non-semisimple category of weight modules for the quantum group. We prove a Verlinde formula which allows for the computation of dimensions and Euler characteristics of topological field theory state spaces of unmarked surfaces. When $r$ is congruent to $\pm 1$ or $\pm 2$ modulo $8$, we relate the resulting $3$-manifold invariants with physicists' $\widehat{Z}$-invariants associated to $\mathfrak{osp}(1 \vert 2)$. Finally, we establish a relation between $\widehat{Z}$-invariants associated to $\mathfrak{sl}(2)$ and $\mathfrak{osp}(1 \vert 2)$ which was conjectured in the physics literature.
\end{abstract}

\maketitle

\tableofcontents

\section*{Introduction}
\addtocontents{toc}{\protect\setcounter{tocdepth}{1}}

The first goal of this paper is to develop the representation theory of an unrolled quantization of the orthosymplectic Lie superalgebra $\osp$. The second goal is to connect this representation theory to the non-semisimple quantum topology of $3$-manifolds. We achieve the second goal in two ways. First, we construct a three-dimensional non-semisimple topological field theory (TFT) $\TQFT$ using the framework of relative modular categories. Second, we study the $\Zhat^{\mathfrak{g}}$-invariants of $3$-manifolds, which were recently introduced by Gukov, Pei, Putrov and Vafa. Here $\mathfrak{g}$ is a complex semisimple Lie superalgebra. More precisely, we establish direct relations between $\Zhat^{\osp}$ and $\Zhat^{\sltwo}$ and between $\Zhat^{\osp}$ and the $3$-manifold invariants defined by $\TQFT$. In the remainder of the introduction, we explain our results in more detail and explain the general context into which they fit.

Let $\osp$ be the complex orthosymplectic Lie superalgebra associated to the super vector space $\C^{1 \vert 2}$ with the canonical non-degenerate supersymmetric bilinear form. The Lie superalgebra $\osp$ is in many respects the simplest Lie superalgebra. For example, it is basic\footnote{Recall that a Lie superalgebra is \emph{basic} if admits a non-degenerate invariant (even) bilinear form and its even subalgebra is reductive.} classical of rank one and its category of finite dimensional complex representations is semisimple.\footnote{In fact, the category of finite dimensional representations of a Lie superalgebra $\mathfrak{g}$ is semisimple if and only $\mathfrak{g}$ is a semisimple Lie algebra or is isomorphic to $\mathfrak{osp}(1 \vert 2n)$ for some $n \geq 1$ \cite{djokovic1987}.} The quantum group of $\osp$ and its representation theory were originally studied in \cite{kulish1988,kulish1989,kulish1989b} with applications to quantum topology following shortly thereafter. A modified version of the Reshetikhin--Turaev construction of link and $3$-manifold invariants \cite{reshetikhin1990,reshetikhin1991} which applies to certain ribbon Hopf superalgebras was introduced and studied in the context of the small quantum group of $\osp$ at primitive roots of unity of odd order \cite{zhang1994,zhang1995,lee1996}. This approach utilizes only a small class of representations of $\osp$, requiring, in particular, that all quantum dimensions are non-vanishing. In various settings, relations between link invariants associated to quantizations of $\mathfrak{osp}(1 \vert 2n)$ and $\mathfrak{so}(2n+1)$---again coloured by only a small class of representations---are known \cite{rittenberg1982,zhang1992,blumen2010,clark2017}.

The constructions of this paper are fundamentally different than those of the previous paragraph. The key new feature is the consideration of a much larger class of representations of quantum $\osp$, in particular those of quantum dimension zero. Not only does this allow the definition of a larger class of link invariants, the resulting $3$-manifold invariants are shown to be the top level of a three-dimensional TFT. Another feature of our approach is the systematic treatment of all roots of unity $e^{\frac{2 \pi \I}{r}}$, $r \geq 3$, without regard for the parity of $r$. In this way, interesting behaviour depending on the congruence class of $r$ modulo $8$ is revealed.

We now explain our results in more detail. The central algebraic object of the paper is the restricted unrolled quantum group $\Uq$, an infinite dimensional Hopf superalgebra obtained as a semi-direct product of the restricted quantum group of $\osp$ with the group algebra of a Cartan lattice. Here $q=e^{\frac{2 \pi \I}{r}}$ for an integer $r \geq 3$. Section \ref{sec:unrolledOsp} is devoted to studying the representation theory of $\Uq$ and is presented so as to ease comparison with the representation theory of $\overline{U}_q^H(\mathfrak{sl}(2))$, as developed in \cite{costantino2015}. The category $\cat$ of finite dimensional $\Uq$-weight modules is rigid monoidal abelian but is neither finite nor semisimple. Motivated by the known universal $R$-matrix for the $\hbar$-adic quantum group of $\osp$ \cite{kulish1989,saleur1990}, we construct in Proposition \ref{prop:braiding} a braiding on the category $\cat$. Surprisingly, with respect to the natural class of pivotal structures, the category $\cat$ is ribbon only when $r \not\equiv 4 \mod 8$; see Proposition \ref{prop:ribbonCat}. In future work, we give an independent, Hopf algebraic, perspective on the lack of ribbon structure when $r \equiv 4 \mod 8$.

Our first main result, which is the culmination of our study of the representation theory of $\Uq$, can now be stated as follows. See Section \ref{sec:relModCat} for recollections on relative modular categories.

\begin{thm}[{Theorem \ref{thm:relMod}}]
\label{thm:relModIntro}
If $r \not\equiv 4 \mod 8$, then the category $\cat$ admits a relative modular structure.
\end{thm}

The details of the relative modular structure, which includes a grading $\cat = \bigoplus_{g \in \Gr} \cat_g$ by an abelian group $\Gr$ and an abelian group $\FR \simeq \FR_0 \times \Ztwo$ together with a ribbon functor $\sigma: \FR \rightarrow \cat_0$, depend on the congruence class of $r$ modulo $8$. As explained below, the cases $r \equiv 1 \mod 2$ and $r \equiv 2 \mod 4$ should be seen as belonging to a single family while that of $r \equiv 0 \mod 8$ is a distinct family. By general results of De Renzi \cite{derenzi2022}, a relative modular structure on $\cat$ gives rise to a symmetric monoidal functor
\[
\TQFT_{\cat} : \Cob_{\cat} \rightarrow \ZVect_{\C},
\]
the three-dimensional oriented non-semisimple TFT associated to $\cat$. The domain $\Cob_{\cat}$ is a category of decorated surfaces and their admissible bordisms. For example, morphisms include the data of a  cohomology class $\coh$ on the underlying $3$-manifold with coefficients in $\Gr$. The codomain $\ZVect_{\C}$ is the monoidal category of $\FR$-graded vector spaces with a symmetric braiding which, upon restricting to underlying $\Ztwo$-graded vector spaces, recovers that of the category of super vector spaces. The $3$-manifold invariants defined by $\TQFT_{\cat}$ are, up to normalization, the CGP invariants introduced by first author, Geer and Patureau-Mirand \cite{costantino2014}.

Theorem \ref{thm:relModIntro} adds to the growing list of relative modular categories, and so non-semisimple TFTs, which so far includes weight modules over unrolled quantum groups associated to simple Lie algebras \cite{blanchet2016,derenzi2020}, the Lie superalgebras $\mathfrak{sl}(m \vert n)$, $m \neq n$, \cite{anghel2021,ha2022}, $\mathfrak{gl}(1\vert 1)$ \cite{geerYoung2025} and Lie superalgebras with abelian bosonic subalgebra \cite{garnerGeerYoung2025}. The resulting TFTs are less developed, with calculations being limited to $\sltwo$ \cite{blanchet2016}, $\mathfrak{gl}(1 \vert 1)$ \cite{geerYoung2025} and Lie superalgebras with abelian bosonic subalgebra \cite{garnerGeerYoung2025}.

In Section \ref{sec:tft} we perform a number of calculations for $\TQFT_{\cat}$. We begin by proving a Verlinde formula, Theorem \ref{thm:verlinde}, which relates the partition function of a trivial circle bundle over a surface $\Sigma_g$ of genus $g \geq 1$, with $\coh$ having holonomy $\beta \in \Gr$ along the circle fibre, to a $\beta$-dependent specialization of the generating function of $\FR$-graded dimensions of $\TQFT_{\cat}(\Sigma_g)$. Our derivation of the Verlinde formula follows previous derivations \cite{blanchet2016,geerYoung2025,garnerGeerYoung2025}, relying on an explicit surgery presentation of the $3$-manifold in question and the computation of modified quantum dimensions of projective $\Uq$-modules. The Verlinde formula is a key tool to establish the following result, which we state for $g \geq 2$.

\begin{thm}[{Corollaries \ref{cor:VerlindeEuler} and \ref{cor:VerlindeTotalDim}}]
\label{thm:stateSpacesIntro}
The Euler characteristic with respect to the parity subgroup $\Ztwo$ of $\FR$ and total dimension of the state space $\TQFT_{\cat}(\Sigma_g)$ are given by
\[
\chi(\TQFT_{\cat}(\Sigma_g))
=
\begin{cases}
0 & \mbox{if } r \equiv 1 \mod 2, \\
0 & \mbox{if } r \equiv 2 \mod 4, \\
\frac{r^{3g-3}}{2^{2g-3}}& \mbox{if } r \equiv 0 \mod 8
\end{cases}
\]
and
\[
\dim_{\C} \TQFT_{\cat}(\Sigma_g)
=
r^{3g-3} \cdot \begin{cases}
2^{2g-2} & \mbox{if } r \equiv 1 \mod 2, \\
\frac{1}{2^{g-1}} & \mbox{if } r \equiv 2 \mod 4, \\
\frac{1}{2^{2g-3}} & \mbox{if } r \equiv 0 \mod 8.
\end{cases}
\]
\end{thm}

The computation of the Euler characteristic follows by taking the limit $\beta \rightarrow 0$ in the Verlinde formula. The total dimension is more difficult to access since it does not obviously arise as a limit of the Verlinde formula. To resolve this, we provide in Theorem \ref{thm:genusgStateSpace} a combinatorially defined spanning set of $\TQFT_{\cat}(\Sigma_g)$ which, in particular, constrains the $\FR$-support of $\TQFT_{\cat}(\Sigma_g)$. Using these support conditions, the total dimension is seen to arise as a limit of the Verlinde formula, leading to the above result.

In Section \ref{sec:Zhat} we shift attention to the $\Zhat$-invariants (also called homological blocks) recently introduced in the physics literature in the context of three-dimensional $\mathcal{N}=2$ supersymmetric gauge theory with the goal of categorifying the Reshetikhin--Turaev invariants of $3$-manifolds \cite{gukov2017,gukov2020}. Physically, $\Zhat^{\mathfrak{g}}(M;\qq)$ is the BPS index of a system of intersecting fivebranes wrapping $M$ in M-theory. These are formal series $\Zhat_{\mathfrak{s}}^{\mathfrak{g}}(M;\qq)$ in an indeterminate $\qq$ which depend on a complex Lie superalgebra $\mathfrak{g}$ and geometric structure $\mathfrak{s}$ on the $3$-manifold $M$. At present, there is a mathematical definition of $\Zhat_{\mathfrak{s}}^{\mathfrak{g}}(M;\qq)$ only for certain $3$-manifolds $M$, such as weakly negative definite plumbed $3$-manifolds \cite{gukov2020,gukov2021b}; see Section \ref{sec:plumbed3Mfld}. While the case $\mathfrak{g}=\sltwo$ is most studied, more general $\mathfrak{g}$ have recently received more attention \cite{park2020,chung2022,chauhan2023,chauhan2024,ferrari2024,moore2024}. In this paper we are primarily interested in the case $\mathfrak{g}=\osp$, where again $\mathfrak{s}$ is a $\Spinc$ structure. Physically, this case results from the inclusion of orientifold planes in the M-theoretic interpretation above.

Our next result is a precise relation between $\Zhat$-invariants for $\sltwo$ and $\osp$.

\begin{thm}[{Theorem \ref{thm:sltwoOspZhat}}]
\label{thm:sltwoOspZhatIntro}
Let $\Gamma$ be a weakly negative definite plumbing graph, $M$ the rational homology $3$-sphere obtained from integral surgery on the framed link defined by $\Gamma$ and $\mathfrak{s}$ a $\Spinc$ structure on $M$. There exists a root of unity $c_{\mathfrak{s}} \in \C$ whose order is divisible by $8 \vert H_1(M;\Z)\vert$ such that
\[
\Zhat^{\osp}_{\mathfrak{s}}(M; -\qq)
=
c_{\mathfrak{s}} \Zhat_{\mathfrak{s}}^{\sltwo}(M; \qq).
\]
\end{thm}

Theorem \ref{thm:sltwoOspZhatIntro}, whose proof is a direct calculation using the surgery definition of $\Zhat$-invariants, confirms the expected relation between $\Zhat^{\sltwo}$ and $\Zhat^{\osp}$ \cite{chauhan2023} and continues a long line of known relations between quantum invariants associated to $\sltwo$ and $\osp$ and, more generally, $\mathfrak{so}(2n+1)$ and $\mathfrak{osp}(1 \vert 2n)$ \cite{zhang1992,ennes1998,blumen2010,clark2017}.

Finally, we establish a relation between the $3$-manifold invariants associated to the TFT $\TQFT_{\cat}$ and $\Zhat^{\osp}$-invariants. More precisely, we use a slight renormalization of the former, which we denote by $N_r^{\osp}$; see Definition \ref{def:cgpInvtOrig}.

\begin{thm}[{Theorem \ref{thm:cgpVsZhat}}]
\label{thm:cgpVsZhatIntro}
Let $\delta \in \{\pm 1\}$. Let $\Gamma$ be a weakly negative definite plumbing graph, $M$ the rational homology $3$-sphere obtained from integral surgery and $\coh \in H^1(M; \Gr)$ a sufficiently generic cohomology class. Under the technical assumptions of Hypothesis \ref{hyp:techAssump}, there is an equality
\[
N_r^{\osp}(M,\coh)
=
\lim_{\qq \rightarrow e^{\frac{4\pi \I}{r}}} \sum_{\mathfrak{s} \in \Spinc(M)} c^{\osp}_{\coh,\mathfrak{s}} \Zhat^{\osp}_{\mathfrak{s}} (M;\qq),
\]
where
\[
c^{\osp}_{\coh,\sigma(b,s)}
=
\frac{e^{\pi{\I} \mu(M,s)} \tor(M,[4 \coh])}{\vert H_1(M;\Z) \vert} \sum_{a,f} e^{2 \pi \I \left(-\frac{r-\delta}{8} \lk(a,a) - \lk(a,b+f) +2 \lk(f,f)-\frac{1}{2} \coh(a) \right)}
\]
if $r \equiv \delta \mod 8$ and
\[
c^{\osp}_{\coh,\sigma(b,s)}
=
\frac{e^{-\delta \frac{\pi \I}{2} \mu(M,s)} \tor(M,[2 \coh])}{\vert H_1(M;\Z) \vert} \sum_{a,f} e^{2 \pi \I \left(-\frac{r-2\delta}{8} \lk(a,a) - \lk(a,b+\delta f) + \delta \lk(f,f) - \frac{1}{2} \coh(a) \right)}
\]
if $r \equiv 2 \delta \mod 8$.
In the above formulae,
\begin{itemize}
\item $\sum_{a,f}$ denotes summation over $a,f \in H_1(M;\Z)$,
\item $\lk : H_1(M;\Z) \times H_1(M;\Z) \rightarrow \Q \slash \Z$ is the linking pairing,
\item $\sigma(b,s)$ is the $\Spinc$ structure associated to $b \in H_1(M;\Z)$ and $\Spin$ structure $s$,
\item$\mu$ is the Rokhlin invariant, and
\item $\tor$ is the appropriately normalized Reidemeister torsion.
\end{itemize}
\end{thm}

Analogues of Theorem \ref{thm:cgpVsZhatIntro} are known for $\sltwo$ \cite{costantino2023} and $\mathfrak{sl}(2 \vert 1)$ \cite{ferrari2024}, again for particular roots of unity. As in these cases, the proof of Theorem \ref{thm:cgpVsZhatIntro} is a delicate sequence of applications of reciprocity of Gauss sums. As discussed in Remark \ref{rem:ZhatVsCGP}, the cases $r \equiv \pm 2 \mod 8$ of Theorem \ref{thm:cgpVsZhatIntro} bear a strong resemblance to the corresponding cases for $\sltwo$ \cite{costantino2023}. The cases $r \equiv \pm 1 \mod 8$ do not have a counterpart in \cite{costantino2023} since in \emph{loc. cit.} only roots of unity of even order are considered. As explained in Sections \ref{sec:pm3Mod8} and \ref{sec:0Mod8}, when $r \equiv 0, \pm 3 \mod 8$ the calculations involved in proving Theorem \ref{thm:cgpVsZhatIntro} do not lead to a universal topological relation between $N_r^{\osp}$ and $\Zhat^{\osp}$. The remaining case studied in \cite{costantino2023}, namely $r \equiv 4 \mod 8$, does not relate to the present work since it is in precisely this case that the category $\cat$ is not ribbon. In view of \cite{costantino2023}, it is natural to expect Theorem \ref{thm:cgpVsZhatIntro} to hold, with minor modifications, for more general closed oriented $3$-manifolds, although we do not pursue this generalization.

Theorem \ref{thm:cgpVsZhatIntro} is another instance of the relation between $\Zhat$-invariants and previously defined quantum invariants of $3$-manifolds. A relation between Reshetikhin--Turaev invariants for the small quantum group of $\sltwo$ and $\Zhat^{\sltwo}$-invariants was conjectured in \cite{gukov2017,gukov2020}. For proofs of this conjecture, in various special cases, and its generalization to other $\mathfrak{g}$, see \cite{fuji2021,andersen2022,mori2022,costantino2023,chauhan2023,murakami2024,chauhan2024}.

To close this introduction, we mention two interesting problems which we do not address in this paper. The first is to provide a conceptual reason for the observed close similarity between the CGP and $\Zhat$-invariants associated to $\sltwo$ and $\osp$. In the semisimple setting, the quantum covering groups of Clark, Hill and Wang \cite{clark2013b} provide such an reason \cite{clark2017}. The second is to find a physical realization of the TFT $\TQFT_{\cat}$ using as a guide recent successes in this direction for other Lie (super)algebras \cite{gukov2021,creutzig2024,geerYoung2025,garnerGeerYoung2025}.

\subsection*{Acknowledgements}
The authors thank Nathan Geer, Thomas Kerler, and Cris Negron for discussions. The authors also thank the anonymous referees for comments and corrections. A.\,R. is partially supported by National Science Foundation grants DMS-2104497 (Geer) and DMS-2302363 (Young). F.\,C. is supported by CIMI Labex ANR 11-LABX-0040 at IMT Toulouse within the program ANR-11-IDEX-0002-02. M.\,H. is partially supported through the NSF-RTG grant DMS-2135960. M.\,B.\,Y. is partially supported by National Science Foundation grants DMS-2302363 and DMS-2440471 and a Simons Foundation Collaboration Grant for Mathematicians (Award ID 853541).

\section{Preliminary material}

The ground field is $\C$ unless stated otherwise.

\subsection{Superalgebra}
\label{sec:supAlg}

Our superalgebra conventions match \cite[I-Supersymmetry]{deligne1999}. The parity of a homogeneous element $v$ of a super vector space $V=V_{\p 0} \oplus V_{\p 1}$ is denoted $\p v \in \Ztwo$. Morphisms of super vector spaces are parity preserving linear maps. The tensor product of super vector spaces $V$ and $W$ is the tensor product of the underlying vector spaces with $\Ztwo$-grading given by
\[
(V \otimes W)_{\p p} = \bigoplus_{\p a + \p b = \p p} V_{\p a} \otimes W_{\p b}.
\]
A (left) module over a (associative) superalgebra $A$ is a super vector space $M$ with an $A$-module structure for which the action map $A \otimes M \rightarrow M$ is a morphism of super vector spaces. Given homogeneous elements $a,b \in A$, set $[a,b] = ab -(-1)^{\p a \p b} ba$.

\subsection{Monoidal categories}
\label{sec:monCat}

Our conventions for monoidal categories match \cite{etingof2015}.

Let $\cat$ be a $\C$-linear abelian monoidal category. We assume that the functor $\otimes : \cat \times \cat \rightarrow \cat$ is $\C$-bilinear, the monoidal unit $\unit \in \cat$ is simple and the $\C$-algebra map $\C \to \End_{\cat}(\unit), k \mapsto k \cdot \id_\unit$, is an isomorphism. If $\cat$ is in addition rigid, braided and has a compatible twist, then $\cat$ is called a \emph{$\C$-linear ribbon category}. The left and right duality structure maps are
\begin{equation*}
\tev_V:V^{\vee} \otimes V \to \unit,
\qquad \tcoev_V:\unit \to V  \otimes V^{\vee}
\end{equation*}
and
\begin{equation*}
\qquad \ev_V:V\otimes V^{\vee}  \to \unit,
\qquad \coev_V: \unit \to V^{\vee}  \otimes V, \end{equation*}
respectively, the braiding is $c=\{c_{V,W}: V \otimes W \rightarrow W \otimes V\}_{V,W \in \cat}$ and the twist is $\theta=\{\theta_V: V \rightarrow V\}_{V \in \cat}$. An object $V\in\cat$ is \emph{regular} if $\tev_V$ is an epimorphism. 

In diagrammatic computations with ribbon categories, we read diagrams from left to right and bottom to top. Basic morphisms are
\[
\id_V
=
\begin{tikzpicture}[anchorbase]
\draw[->,thick] (0,0) -- node[left] {\small$V$} (0,1);
\end{tikzpicture}
\qquad ,\qquad
\id_{V^{\vee}}
=
\begin{tikzpicture}[anchorbase]
\draw[<-,thick] (0,0) -- node[left] {\small$V$} (0,1);
\end{tikzpicture}
\]
\[
\tev_V
=
\begin{tikzpicture}[anchorbase]
\draw[->,thick] (0,0)  arc (0:180:0.5 and 0.75);
\node at (-1.4,0)  {$V$};
\end{tikzpicture}
\qquad, \qquad
\tcoev_V
=
\begin{tikzpicture}[anchorbase]
\draw[<-,thick] (0,0)  arc (180:360:0.5 and 0.75);
\node at (-0.5,-0.1)  {$V$};
\end{tikzpicture}
\]
\[
\ev_V
=
\begin{tikzpicture}[anchorbase]
\draw[<-,thick] (0,0)  arc (0:180:0.5 and 0.75);
\node at (0.4,0)  {$V$};
\end{tikzpicture}
\qquad, \qquad
\coev_V
=
\begin{tikzpicture}[anchorbase]
\draw[->,thick] (0,0)  arc (180:360:0.5 and 0.75);
\node at (1.5,-0.1)  {$V$};
\end{tikzpicture}
\]
\[
c_{V,W}
=
\begin{tikzpicture}[anchorbase]
\draw[->,thick] (0.5,0) -- node[right,near start] {\small $W$} (0,1);
\draw[->,thick,cross line] (0,0) -- node[left,near start] {\small$V$} (0.5,1);
\end{tikzpicture}
\qquad, \qquad
\theta_V
=
\begin{tikzpicture}[anchorbase]
\draw[->,thick,rounded corners=8pt] (0.25,0.25) -- (0,0.5) -- (0,1);
\draw[thick,rounded corners=8pt,cross line] (0,0) -- (0,0.5) -- (0.25,0.75);
\draw[thick] (0.25,0.75) to [out=30,in=330] (0.25,0.25);
\node at (-0.2,0.2)  {$V$};
\end{tikzpicture}.
\]

Let $\mathcal{R}_{\cat}$ be the ribbon category of $\cat$-coloured ribbon graphs in $\R^2 \times [0,1]$ and $F_{\cat} : \mathcal{R}_{\cat} \rightarrow \cat$ the associated Reshetikhin--Turaev functor \cite[Theorem I.2.5]{turaev1994}. Two formal $\C$-linear combinations of $\cat$-coloured ribbon graphs are \emph{skein equivalent} if their images under $F_{\cat}$ agree. The corresponding equivalence relation is denoted $\dot{=}$.

\subsection{Relative modular categories}
\label{sec:relModCat}

We recall a number of definitions from \cite{geer2011,costantino2014,derenzi2022}. Let $\cat$ be a $\C$-linear abelian ribbon category. 

Let $V , W \in \cat$. Recall that right partial trace along $W$ is the map
\begin{eqnarray*}
\ptr_W : \End_{\cat}(V \otimes W) & \rightarrow & \End_{\cat}(V) \\
f & \mapsto & (\id_V \otimes \ev_W) \circ (f \otimes \id_{W^{\vee}}) \circ (\id_V \otimes \tcoev_W).
\end{eqnarray*}

\begin{Def}
A full subcategory $\mathcal{I} \subset \cat$ is an {\em ideal} if it has the following properties:
\begin{enumerate}
\item If $U\in \mathcal{I}$ and $V \in \cat$, then $U\otimes V \in \mathcal{I}$.
\item If $U \in \mathcal{I}$ and $V \in \cat$ and there exist morphisms $f:V\to U$ and $g:U\to V$ satisfying $g \circ f=\id_{V}$, then $V\in \mathcal{I}$.
\end{enumerate}
\end{Def}

\begin{Def}
\label{def:mtrace}
\begin{enumerate}
\item A \emph{modified trace} on an ideal $\mathcal{I} \subset \cat$ is a family of $\C$-linear functions 
$\mt=\{\mt_V:\End_\cat(V) \rightarrow \C \mid V \in \mathcal{I} \}$ with the following properties:
\begin{enumerate}
\item \emph{Cyclicity}: For all $V,W \in \mathcal{I}$ and $f \in \Hom_{\cat}(W,V)$ and $g \in \Hom_{\cat}(V,W)$, there is an equality $\mt_V(f \circ g)=\mt_W(g \circ f)$.
\item \emph{Partial trace property}: For all $V \in \mathcal{I}$, $W \in \cat$ and $f\in\End_{\cat}(V\otimes W)$, there is an equality $\mt_{V\otimes W}(f)=\mt_V(\ptr_W(f))$.
\end{enumerate}

\item The \emph{modified dimension} of $V\in \mathcal{I}$ is $\qd(V)=\mt_V(\id_V)$.
\end{enumerate}
\end{Def}

\begin{Def}
\begin{enumerate}
\item A set $\mathcal{E}=\{ V_i \mid i \in J \}$ of objects of $\cat$ is \emph{dominating} if for any $V \in \cat$ there exist indices $\{i_1,\dots,i_m \} \subseteq J $ and morphisms $\iota_k \in \Hom_{\cat}(V_{i_k},V)$, $s_k \in \Hom_{\cat}(V,V_{i_k})$ such that $\id_{V}=\sum_{k=1}^m \iota_k \circ s_k$.

\item A dominating set $\mathcal{E}$ is \emph{completely reduced} if $\dim_\C \Hom_{\cat}(V_i,V_j)=\delta_{ij}$ for all $i,j \in J$.
\end{enumerate}
\end{Def}

Let $(\FR,+)$ be an abelian group. We often view $\FR$ as a discrete monoidal category with object set $\FR$.

\begin{Def}
\label{def:free}
A \emph{free realization of $\FR$ in $\cat$} is a monoidal functor $\sigma: \FR \rightarrow \cat$, $k \mapsto \sigma_k$, such that
\begin{enumerate}
\item $\sigma_0=\mathbb{I}$,
\item $\theta_{\sigma_k}=\id_{\sigma_k} \text{ for all } k \in \FR$, and
\item if $V \otimes \sigma_k \simeq V$ for a simple object $V \in \cat$, then $k=0$.  
\end{enumerate}
\end{Def}

We often identify $\sigma: \FR \rightarrow \cat$ with the set of objects $\sigma_{\FR} := \{\sigma_k \mid k \in \FR\}$, omitting from the notation the monoidal coherence data.

\begin{Def}
\label{def:Gstr}
Let $(\Gr,+)$ be an abelian group.
\begin{enumerate}
\item A subset $\SSS \subset \Gr$ is \emph{symmetric} if $\SSS=-\SSS$ and \emph{small} if $\bigcup_{i=1}^n (g_i+\SSS) \neq \Gr$ for all $g_1,\ldots ,g_n\in \Gr$.

\item A \emph{$\Gr$-grading on $\cat$} is an equivalence of $\C$-linear abelian categories $\cat \simeq \bigoplus_{g \in \Gr} \cat_g$, where $\{\cat_g \mid g \in \Gr\}$ are full subcategories of $\cat$, which has the following properties:
\begin{enumerate}
\item $\mathbb{I} \in \cat_0$.
\item If $V\in\cat_g$, then $V^{\vee}\in\cat_{-g}$.
\item If $V\in\cat_g$ and $V^{\prime} \in \cat_{g^{\prime}}$, then $V\otimes V^{\prime}\in\cat_{g+g^{\prime}}$.
\end{enumerate}

\item A $\Gr$-graded abelian category $\cat$ is \emph{generically semisimple} if there exists a small symmetric subset $\SSS \subset \Gr$ such that each subcategory $\cat_g$, $g \in \Gr \setminus \SSS$, is semisimple.
\end{enumerate}
\end{Def}

\begin{Def}
\label{def:preMod}
Let $\Gr$ and $\FR$ be abelian groups, $\SSS \subset \Gr$ a small symmetric subset and $\cat$ a $\C$-linear abelian ribbon category. Suppose that the following data is given:
\begin{enumerate}
\item A $\Gr$-grading on $\cat$.
\item A free realization $\sigma : \FR \rightarrow \cat_0$.
\item A non-zero modified trace $\mt$ on the ideal of projective objects of $\cat$.
\end{enumerate}
Call $\cat$ a \emph{pre-modular $\Gr$-category relative to $(\FR,\SSS)$} if it has the following properties:
\begin{enumerate}
\item \emph{Generic semisimplicity}: \label{def:genSS} For each $g \in \Gr \setminus \SSS$, there exists a finite set of regular simple objects $\Theta(g):=\{ V_i \mid i \in I_g  \}$ such that
$$\Theta(g) \otimes \sigma_{\FR}:=\{ V_i \otimes \sigma_k \mid i \in I_g, \; k\in \FR  \}$$
is a completely reduced dominating set for $\cat_g$.

\item
\label{def:compat}
There exists a bicharacter $\psi: \Gr \times \FR \rightarrow \C^{\times}$ such that
  \begin{equation*}
    \label{eq:psi}
    c_{\sigma_k,V}\circ c_{V,\sigma_k}= \psi(g,k) \cdot  \id_{V \otimes \sigma_k}
  \end{equation*}
for all $g\in \Gr$, $V \in \cat_g$ and $k \in \FR$.
\end{enumerate}
\end{Def}
 
\begin{Def}
\label{def:ndeg}
Let $\cat$ be a pre-modular $\Gr$-category relative to $(\FR,\SSS)$.
\begin{enumerate}
\item The \emph{Kirby colour of index $g \in \Gr \setminus \SSS$} is $\Omega_g:= \sum_{V \in \Theta(g)}\qd(V) \cdot V$.
\item The \emph{stabilization coefficients} $\Delta_\pm\in \C$ are defined by the skein equivalences
$$\epsh{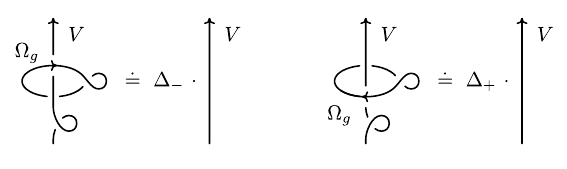}{15ex}
$$
for any $g \in \Gr \setminus \SSS$ and $V\in \cat_g$.
\item The pre-modular $\Gr$-category $\cat$ is \emph{non-degenerate} if $\Delta_{+}\Delta_{-}\neq 0$. 
\end{enumerate}
\end{Def}

\begin{Def}
\label{def:modG}
A \emph{modular $\Gr$-category relative to $(\FR,\SSS)$} is a pre-modular $\Gr$-category $\mathcal{C}$ relative to $(\FR,\SSS)$ for which there exists a scalar $\zeta \in \C^{\times}$, the \emph{relative modularity parameter}, such that
\begin{equation}\label{eq:mod}
    \epsh{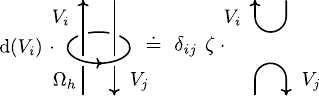}{10ex}
\end{equation}
for all $g,h \in \Gr \setminus \SSS$ and $i,j \in I_g$.
\end{Def}

The relative modularity parameter satisfies $\zeta = \Delta_+ \Delta_-$ \cite[Proposition 1.2]{derenzi2022}. In particular, relative modular categories are non-degenerate relative pre-modular.

\subsection{\texorpdfstring{$\Spin$ and $\Spinc$ structures on $3$-manifolds}{Spin and Spinc structures on 3-manifolds}}
\label{sec:surgeryTopology}

Following \cite[\S 2.2]{deloup2005}, we recall basic topology of $3$-manifolds obtained by surgery on links in $S^3$.

Let $L \subset S^3$ be a framed oriented link with connected components $V$ and $\vert V \vert \times \vert V \vert$ linking matrix $B$. Let $M$ be the closed oriented $3$-manifold obtained by performing integral surgery on $L$. For $\Gr$ an abelian group, there are canonical isomorphisms
\begin{equation}
\label{eq:firstHom}
H_1(M; \Z)
\simeq
\Z^{V} \slash B \Z^{V}
\end{equation}
and
\begin{equation}
\label{eq:firstCohom}
H^1(M;\Gr)
\simeq 
\{\phi \in \Gr^{V} \vert B \phi =0 \}.
\end{equation}
Here we write $\Z^V$ for the free abelian group on the set $V$ and view its elements as column vectors. Throughout the paper we assume that $B$ is non-singular, so that $M$ is a rational homology sphere. In this case, $\vert \det B \vert = \vert H_1(M;\Z) \vert$ and, with respect to the isomorphism \eqref{eq:firstHom}, the linking pairing is
\begin{equation*}
\label{eq:linkingNum}
\lk: H_1(M;\Z) \otimes H_1(M;\Z) \rightarrow \Q \slash \Z,
\qquad
a \otimes b \mapsto a^t B^{-1} b \mod 1
\end{equation*}
where $(-)^t$ is transposition.

Denote by $\Spin(M)$ and $\Spinc(M)$ the sets of equivalence classes of $\Spin$ and $\Spinc$ structures on $M$, respectively. There are identifications
\[
\Spin(M)
=
\{s \in \left(\Z\slash 2 \Z \right)^{V} \mid \sum_{j \in V} B_{ij} s_j \equiv B_{ii} \mod 2 \mbox{ for all } i \in V\}
\]
and
\[
\Spinc(M)
=
\{K \in \Z^{V} \slash 2 B \Z^{V} \mid K_i \equiv B_{ii} \mod 2 \mbox{ for all } i \in V\}.
\]
There is a canonical surjective map
\[
\sigma: H_1(M;\Z) \times \Spin(M) \rightarrow \Spinc(M),
\qquad
\sigma(b,s) = 2b+ i(s),
\]
where $i(s) = B \tilde{s}$ for any lift $\tilde{s} \in \Z^{V}$ of $s \in \left(\Z\slash 2 \Z \right)^{V}$. Finally, given $s \in \Spin(M)$, the Rokhlin invariant $\mu(M,s) \in \Z \slash 16 \Z$ of $M$ with spin structure $s$ satisfies
\begin{equation}
\label{eq:RokhlinMod4}
\mu(M,s) \equiv \sigma - s^t B s \mod 4,
\end{equation}
where $\sigma$ is the signature of $B$.

\section{The category of weight \texorpdfstring{$\Uq$}{Uq(osp(1|2))}-modules}
\label{sec:unrolledOsp}

\subsection{Super quantum integers}
\label{sec:supInt}

We introduce super variants of the standard quantum integers. Our conventions differ slightly from those in the literature (\emph{cf.} \cite[\S 2]{kulish1989b}, \cite[\S 2.5]{clark2013}).

Let $v$ be an indeterminate and $\mathcal{A}=\Z[\q,\q^{-1}]$. For an integer $n \geq 0$, set
\[
\label{eq:superQuantIntPoly}
\pser{n}
=
\frac{\q^{-n} - (-\q)^n}{\q+\q^{-1}}
=
\sum_{i=0}^{n-1} (-1)^{n+1+i} \q^{n-1-2i} \in \mathcal{A}
\]
and
\[
\pser{n} !=
\begin{cases}
\prod_{i=1}^n \pser{i} & \mbox{if } n \geq 1, \\
1 & \mbox{if } n =0
\end{cases},
\qquad
\pserbin{n}{k} =
\begin{cases}
\frac{\pser{n}!}{\pser{n-k}!\pser{k}!} & \mbox{if } 0 \leq k \leq n, \\
0 & \mbox{if } k <0.
\end{cases}
\]

\begin{Lem}
\label{lem:superPascal}
    For integers $1 \leq k \leq n$, there are equalities
    \[
    \pserbin{n+1}{k}=(-\q)^{n-k+1}\pserbin{n}{k-1}+\q^{-k}\pserbin{n}{k},
    \]
    \[
    \pserbin{n+1}{k}=\q^{k-n-1}\pserbin{n}{k-1}+(-\q)^{k}\pserbin{n}{k}.
    \]
\end{Lem}

\begin{proof}
This is a direct verification.
\end{proof}


\begin{Lem}
\label{lem:vanSum}
For any integer $n \geq 0$, there is an equality
\[
\sum_{k=0}^n (-1)^{\frac{k(k+1)}{2}} \q^{k(n-1)} \pserbin{n}{k}
=
\begin{cases}
1 & \mbox{if } n=0,\\
0 & \mbox{if } n >0.
\end{cases}
\]
\end{Lem}

\begin{proof}
The equality is proved using Lemma \ref{lem:superPascal} and induction on $n$.
\end{proof}


\subsection{The quantum group of \texorpdfstring{$\osp$}{osp(1|2)}}
\label{sec:UqospGen}

\begin{Def} \label{def:quantum_group}
Let $\UqG$ be the unital superalgebra over $\Q(\q)$ with generators $K$, $K^{-1}$ of parity $\p 0$ and $E$, $F$ of parity $\p 1$ and relations 
\[
KK^{-1}=K^{-1}K=1,
\qquad
KE=\q^2 EK,
\qquad
KF=\q^{-2} FK,
\qquad
[E,F]=\frac{K-K^{-1}}{\q-\q^{-1}}.
\]
\end{Def}

There is a unique Hopf superalgebra structure on $\UqG$ with counit $\epsilon$, coproduct $\Delta$ and antipode $S$ defined on generators by
\[
\epsilon(K)=1,
\qquad
\epsilon(E)=\epsilon(F)=0,
\]
\[
\Delta (K) = K \otimes K,
\qquad
\Delta (E) = 1 \otimes E + E \otimes K,
\qquad
\Delta (F) = K^{-1} \otimes F + F \otimes 1 ,
\]
\[
S(K)=K^{-1},
\qquad
S(E)= -E K^{-1},
\qquad
S(F)= -KF.
\]

For $i,n \in \Z$ with $i \geq 1$, define elements of $\UqG$ by
\[
\pser{K;n}
=
\frac{K \q^n-K^{-1} (-\q)^{-n}}{\q-\q^{-1}},
\qquad
\pserbin{K;n}{i}=\prod_{j=1}^i \dfrac{\pser{K;n+j-i}}{\pser{j}}.
\]
Set $\pserbin{K;n}{0} = 1$. {For an invertible scalar $z$, we define $\pser{z;n}$ and $\pserbin{z;n}{i}$ analogously.} A direct computation shows that
\begin{align} \label{eq:KPserRel}
    (-1)^k\pser{j}\pser{K;i+k} + \pser{k}\pser{K;i-j}
    = 
    \pser{j+k}\pser{K;i}
\end{align}
for all $i,j,k \in \Z$. Compare with \cite[Eqn. (2.8)]{clark2013}.

\begin{Lem}[{\emph{cf}. \cite[Lemma 2.8]{clark2013}}]
\label{lem:EFPow}
For each integer $n \geq 1$, there are equalities
\begin{equation} \label{eq:[E,F^n]}
[E,F^n] = \pser{n} F^{n-1} \pser{K;-n+1},    
\end{equation}
\begin{equation} \label{eq:[F,E^n]}
[F,E^n] = (-1)^n \pser{n} E^{n-1} \pser{K;n-1}.
\end{equation}
\end{Lem}

\begin{proof}
We prove only equation \eqref{eq:[E,F^n]}. We proceed by induction on $n$. The case $n=1$ is a defining relation of $\UqG$. Assuming equation \eqref{eq:[E,F^n]} holds for $n$, we compute
\begin{eqnarray*}
EF^{n+1}
&=&
(-1)^nF^nEF+\pser{n}F^{n-1}\pser{K;-n+1}F \\
&=&
(-1)^{n+1}F^{n+1}E+(-1)^nF^n\pser{K;0}+\pser{n}F^{n}\pser{\q^{-2}K;-n+1}
\end{eqnarray*}
so that
\[
[E,F^{n+1}]
=
F^n\left((-1)^n\frac{K-K^{-1}}{\q-\q^{-1}}+
\frac{\q^{-n} - (-1)^n \q^n}{\q+\q^{-1}}
\cdot
\frac{K \q^{-n-1}+(-1)^{n} K^{-1} \q^{n+1}}{\q-\q^{-1}}
\right),
\]
which is seen to equal $F^n\pser{n+1}\pser{K;-(n+1)+1}$ by a direct calculation.
\end{proof}

\begin{Lem}[{\emph{cf.} \cite[Lemma 4.6]{clark2013}}]
\label{lem:EPow}
For each integer $n \geq 1$, there are equalities
\begin{equation} \label{eq:Delta(En)}
\Delta (E^n)
=
\sum_{j=0}^n \pserbin{n}{j}(-\q)^{j(n-j)} E^j \otimes E^{n-j} K^j,
\end{equation}
\begin{equation*} \label{eq:Delta(Fn)}
\Delta (F^n)
=
\sum_{j=0}^n \pserbin{n}{j}\q^{-j(n-j)} K^{-j} F^{n-j} \otimes F^{j}.
\end{equation*}
\end{Lem}

\begin{proof}
We prove equation \eqref{eq:Delta(En)} by induction on $n$. The case $n=1$ holds by the definition of $\Delta$. Assuming equation \eqref{eq:Delta(En)} holds for $n$, we compute
\begin{eqnarray*}
\Delta (E^{n+1})
&=&
(1 \otimes E + E \otimes K) 
\sum_{j=0}^n \pserbin{n}{j}(-\q)^{j(n-j)} E^j \otimes E^{n-j} K^{j}     \\
&=&
\sum_{j=0}^n \pserbin{n}{j}(-\q)^{j(n-j)}(-1)^j E^j \otimes E^{n+1-j} K^j \\ && +
\sum_{j=0}^n \pserbin{n}{j}(-\q)^{j(n-j)}\q^{2(n-j)} E^{j+1} \otimes E^{n-j} K^{j+1}\\
&=&
1\otimes E^{n+1}+E^{n+1}\otimes K^{n+1}
\\
&&+
\sum_{j=1}^n 
(-\q)^{j(n+1-j)}\left(
\pserbin{n}{j}\q^{-j}
+
\pserbin{n}{j-1}(-\q)^{(n+1-j)}
\right)E^j \otimes E^{n+1-j} K^j 
\\
&=&
\sum_{j=0}^{n+1}\pserbin{n+1}{j}(-\q)^{j(n+1-j)}E^j\otimes E^{n+1-j}K^j,
\end{eqnarray*}
the final equality following from Lemma \ref{lem:superPascal}.
\end{proof}

\subsection{The unrolled quantum group of \texorpdfstring{$\osp$}{osp(1|2)}}
\label{sec:Uqosp}

Let $r \geq 3$ be an integer and $q = e^{\frac{2 \pi \I}{r}}$. We henceforth consider only super quantum integers specialized to $\q = q$, for which we use the same notation as Section \ref{sec:supInt}. Given $z \in \C$, set $q^z = e^{\frac{2 z \pi \I}{r}}$. If $z \neq 0$, define $\pser{z;n} \in \C$ as in Section \ref{sec:UqospGen}, but with $K$ replaced by $z$.

\begin{Def} \label{def:unrolled_quantum_group}
The \emph{unrolled quantum group $\UqUn$} is the unital superalgebra over $\C$ with generators $H$, $K$, $K^{-1}$ of parity $\p 0$ and $E$, $F$ of parity $\p 1$ and relations
\[
[H, K] = [H, K^{-1}] = 0,
\qquad
[H,E]=2E,
\qquad
[H,F]=-2F,
\]
\[
KK^{-1}=K^{-1}K=1,
\qquad
KE=q^2 EK,
\qquad KF=q^{-2} FK,
\qquad
[E,F]=\frac{K-K^{-1}}{q-q^{-1}}.
\]
\end{Def}

\begin{Rem}
The Lie superalgebra $\osp$ is generated by vectors $J, X^+, X^-$ of parity $\p 0$ and $\psi^+,\psi^-$ of parity $\p 1$ with defining relations
\[
[J,X^{\pm}]=\pm X^{\pm},
\qquad
[J,\psi^{\pm}]=\pm \frac{1}{2} \psi^{\pm},
\qquad
[X^+,X^-] = 2 J,
\]
\[
[X^{\pm},\psi^{\pm}]=0,
\qquad
[X^{\pm},\psi^{\mp}]=- \psi^{\pm},
\qquad
[\psi^{\pm},\psi^{\pm}]= \pm 2 X^{\pm},
\qquad
[\psi^+,\psi^-]= 2 J.
\]
In fact, the relations involving only $J$, $\psi^{\pm}$ together with the super Jacobi identify already determine $\osp$. Setting $H = 4 J$, $E = \psi^+$ and $F = 2\psi^-$, this smaller set of relations becomes
\[
[H,E] = 2 E, \qquad [H,F]=-2F,
\qquad
[E, F] = H.
\]
The superalgebra $\UqUn$ is the unrolled quantization of this presentation of $\osp$. This also explains Definition \ref{def:quantum_group}, which agrees with \cite[Example 2.1]{jeong2001} but differs from \cite[\S 2]{kulish1989b} and \cite[Definition 2.1]{clark2013}. Our normalizations are chosen to ease comparison with the representation theory of the unrolled quantum group of $\sltwo$.
\end{Rem}

Give $\UqUn$ a Hopf superalgebra structure by supplementing the specialized Hopf structure of $\UqG$, given in Section \ref{sec:UqospGen}, by the definitions
\[
\epsilon(H)=0,
\qquad
\Delta (H) = H \otimes 1 + 1 \otimes H,
\qquad
S(H)=-H.
\]
Lemmas \ref{lem:EFPow} and \ref{lem:EPow} continue to hold for $\UqUn$.

\begin{Lem}
\label{lem:minVanish}
If $r=4$, then $\pser{n} \neq 0$ for all $n \geq 1$. If $r \neq 4$, then the minimal positive integer $\red{r}$ which satisfies $\pser{\red{r}}=0$ is
\[
\red{r}=
\begin{cases}
2r & \mbox{if } r \equiv 1 \mod 2, \\
r & \mbox{if } r \equiv 2 \mod 4,\\
\frac{r}{2} & \mbox{if } r \equiv 0 \mod 8,\\
\frac{r}{4} & \mbox{if } r \equiv 4 \mod 8.
\end{cases}
\]
\end{Lem}

\begin{proof}
If $r=4$, then $\pser{n} = (-\I)^{n-1} n$. If $r \neq 4$, then the denominator of $\frac{q^{-n} - (-q)^n}{q+q^{-1}}$ is non-zero so that $\pser{n}=0$ if and only if the numerator vanishes. This is equivalent to setting $q^{2n}=(-1)^n$, which holds if and only if $n \in \frac{r}{4} (2\Z + p_n)$ where $p_n \in \{0,1\}$ is the parity of $n$. Note that if $n$ is odd, then $r \equiv 0 \mod 4$. We proceed in cases:
\begin{itemize}
\item If $r \equiv 1 \mod 2$, then $n$ is even, so that $n \in \frac{r}{2} \Z$. It follows that $n=2r$ is the minimal solution.

\item If $r \equiv 2 \mod 4$, then $n$ is even, so that $n \in \frac{r}{2} \Z$. It follows that $n=r$ is the minimal solution.

\item If $r \equiv 0 \mod 8$, then $r=8k$ for $k \in \Z$, so that $n \in 2k(2 \Z + p_n)$. The minimal solution is $n = 4k = \frac{r}{2}$.

\item If $r \equiv 4 \mod 8$, then $r=4k$ with $k \in \Z$ odd, so that $n \in k(2 \Z + p_n)$. The minimal solution is $n = k = \frac{r}{4}$. \qedhere
\end{itemize}
\end{proof}

The peculiar behaviour of $r=4$ in Lemma \ref{lem:minVanish} obstructs the definition of a restricted Hopf algebra quotient of $\UqUn$ with $q=\sqrt{-1}$; see the proof of Lemma \ref{lem:twoSidedHopf} below. For this reason, we henceforth exclude the case $r=4$ from consideration.

For later use, we record that
\begin{equation}
\label{eq:q2r}
q^{2\red{r}}
=
(-1)^{\red{r}}
=
\begin{cases}
-1 & \mbox{if } r \equiv 4 \mod 8, \\
1 & \mbox{otherwise}.
\end{cases}
\end{equation}

Let $\pideal$ and $\mideal$ be the left ideals of $\UqUn$ generated by $E^{\red{r}}$ and $F^{\red{r}}$, respectively, and let $\pmideal=\pideal+\mideal$.

\begin{Lem}
\label{lem:twoSidedHopf}
The left ideal $\pmideal$ is a two-sided Hopf ideal.
\end{Lem}

\begin{proof}
Since $\pser{\red{r}}=0$, equation \eqref{eq:[F,E^n]} gives $[F,E^{\red{r}}]=0$. The defining relations of $\UqUn$ give $K E^{\red{r}} = q^{2 \red{r}} E^{\red{r}} K$ and we conclude that $E^{\red{r}}$ is central up to scalars. It follows that $\pideal$ is a two-sided ideal. Similarly, $\mideal$ is a two-sided ideal. Regarding the Hopf condition, note that the defining property of $\red{r}$ implies
\[
\pserbin{{\red{r}}}{j}=
\begin{cases}
1 &\mbox{if } j\in \{0,{\red{r}}\},\\
0& \mbox{if } 0<j< \red{r}.
\end{cases}
\]
In view of this, Lemma \ref{lem:EPow} gives
\[
\Delta(E^{\red{r}})=E^{\red{r}}\otimes K^{\red{r}}+1\otimes E^{\red{r}}\in \pideal\otimes \UqUn + \UqUn \otimes \pideal.
\]
Using that $S$ is a superalgebra anti-homomorphism, we compute
\[
S(E^{\red{r}})=(-1)^{\frac{\red{r}(\red{r}+1)}{2}} q^{-{\red{r}}({\red{r}}-1)}E^{\red{r}}K^{-\red{r}}\in\pideal.
\]
Similarly, $\mideal$ is Hopf. It follows that $\pmideal$ is a two-sided Hopf ideal.
\end{proof}

\begin{Def} \label{def:restricted_unrolled_quantum_group}
The \emph{restricted unrolled quantum group} is the Hopf superalgebra
\[
\Uq = \UqUn \slash \pmideal.
\]
\end{Def}

The following result asserts that Definition \ref{def:restricted_unrolled_quantum_group} is the only restricted quantum group associated to $\UqUn$, by which we mean that both $E$ and $F$ are nilpotent, no additional relations are imposed on Cartan generators, and the resulting superalgebra is Hopf.

\begin{Prop}
\label{prop:HopfNoGo}
    Let $J\subset \UqUn$ be an ideal which contains $E^{k_1}$ and $F^{k_2}$ for some $k_1,k_2> 0$ and does not contain $K^l$ for $l\neq 0$. Suppose that $n_1$ and $n_2$ are the least such $k_1$ and $k_2$ for which $E^{k_1},F^{k_2}\in J$. Then $J$ is a Hopf ideal if and only if $n_1,n_2\in \{1,\red{r}\}$.
\end{Prop}

\begin{proof}
    One can check that an ideal containing some $E^{k_1}$ and $F^{k_2}$ does not necessarily contain a nonzero power of $K$. To prove the proposition, it suffices to show that $\Delta(E^{n_1})\in J\otimes \UqUn + \UqUn\otimes J$ if and only if $n_1\in \{1,\red{r}\}$. The argument involving the coproduct of $F^{n_2}$ is similar. The case $n_1=1$ is verified immediately. The case $n_1=\red{r}$ is proven in Lemma \ref{lem:twoSidedHopf}. 
    
    Suppose then that $n_1\not\in \{0,1,\red{r}\}$ and write $n_1=a \red{r}+b$ where $a\geq 0$ and $0\leq b <\red{r}$. A straightforward induction shows that 
    \[
    \Delta(E^{a\cdot\red{r}})=
    \sum_{i=0}^{a}\binom{a}{i} E^{i\red{r}}\otimes K^{i\red{r}}E^{(a-i)\red{r}}\,.
    \]
    Thus, by Lemma \ref{lem:EPow}, we have
    \begin{align*}
    \Delta(E^{n_1})&=\Delta(E^{a\cdot\red{r}})\Delta(E^b)\\
    &=
    \left(\sum_{i=0}^{a}\binom{a}{i} E^{i\red{r}}\otimes K^{i\red{r}}E^{(a-i)\red{r}}\right)\left(
\sum_{j=0}^b \pserbin{b}{j}(-\q)^{j(b-j)} E^j \otimes E^{b-j} K^j
\right)
\\&=
\sum_{i=0}^{a}\sum_{j=0}^b\binom{a}{i} \pserbin{b}{j}(-\q)^{j(b-j)} (-1)^{(a-i)j\red{r}}E^{i\red{r}+j}\otimes K^{i\red{r}}E^{(a-i)\red{r}+b-j}K^j\,.
    \end{align*}
    For any $0<i<a$ and $0<j<b$, the exponents $i\red{r}+j$ and $(a-i)\red{r}+b-j$ are less than $n_1$ and nonzero. Moreover, $\pserbin{b}{j}$ is non-zero. Since nonzero powers of $K$ do not belong to the ideal, the corresponding terms in the expansion of $\Delta(E^{n_1})$ do not belong to $J\otimes \UqUn + \UqUn\otimes J$. 
\end{proof}

\begin{Rem}
    Let $q$ be a primitive $r$\textsuperscript{th} root of unity and set
        \[
\tilde{r}=
\begin{cases}
2r & \mbox{if } r \equiv 1 \mod 2, \\
r & \mbox{if } r \equiv 0 \mod 4,\\
\frac{r}{2} & \mbox{if } r \equiv 2 \mod 4.
\end{cases}
\]
In \cite[D.7 Case 1]{blumen2006}, it is claimed that $E$ (denoted by $e_n$ in \emph{loc. cit.}) raised to the power $\tilde{r}$ (denoted by $\overline{N}$ in \emph{loc. cit.}) generates a Hopf ideal of $U_q(\osp)$. The discrepancy with Proposition \ref{prop:HopfNoGo} appears to be due to a misapplication of \cite[Lemma B.0.7]{blumen2006} which incorrectly assumes that $(E\otimes K)(1\otimes E)$ is equal to $-q(1\otimes E)(E\otimes K)$ in the computation of $\Delta (E)^n=(E\otimes K+1\otimes E)^n$. This leads to the incorrect conclusion that $\Delta (E)^{\tilde{r}}$ is equal to
    \[
    \sum_{i=0}^{\tilde{r}}\pserbin{\tilde{r}}{i}_q E^{\tilde{r}-i}\otimes E^iK^{\tilde{r}-i}=
    E^{\tilde{r}}\otimes K^{\tilde{r}}+1\otimes E^{\tilde{r}}.
    \]
    If one instead uses that $(E\otimes K)(1\otimes E)=-q^2(1\otimes E)(E\otimes K)$ and computes $\tilde{s}$ for $q^2$, a primitive root of unity of order $s= \frac{r}{\gcd(2,r)}$, then one recovers $\tilde{s}=\red{r}$ as in Lemma \ref{lem:minVanish}.
\end{Rem}

\subsection{Weight modules}
\label{sec:weightMod}

\begin{Def}
A finite dimensional $\Uq$-module $V$ is a \emph{weight module} if $H$ acts semisimply on $V$ and $K v=q^{\lambda}v$ for any $H$-weight vector $v$ of weight $\lambda \in \C$.
\end{Def}

All $\Uq$-modules in this paper are assumed to be weight. The category $\cat$ of weight $\Uq$-modules and their $\Uq$-linear maps (of parity $\p 0$) is $\C$-linear, abelian and locally finite. The coproduct $\Delta$ induces a monoidal structure on $\cat$ with monoidal unit the trivial module $\unit=\C$.

Given $V \in \cat$, let $V^{\vee} \in \cat$ be the $\C$-linear dual of the underlying super vector space of $V$ with $\Uq$-module structure given by
\[
(x \cdot f)(v) = (-1)^{\p f \p x} f(S(x)v),
\qquad
v \in V, \; f \in V^{\vee}, \; x \in \Uq.
\]
Let $\{v_i\}_i$ be a homogeneous basis of $V$ with dual basis $\{v_i^{\vee}\}_i$. For each integer $s \in \Z$, define putative left and right duality structures on $\cat$ by the $\C$-linear maps
\[
\tev_V(f \otimes v) = f(v),
\qquad
\tcoev_V (1) = \sum_i v_i \otimes v_i^{\vee},
\]
\[
\ev^{(s)}_V (v \otimes f) = (-1)^{\p f \p v}f(K^{1-s} v),
\qquad
\coev^{(s)}_V (1) = \sum_i (-1)^{\p v_i}v_i^{\vee} \otimes K^{s-1} v_i.
\]

\begin{Lem}\label{lem:catqPivot}
If $q^{2s}=1$, then the above maps define a pivotal structure on $\cat$.
\end{Lem}

\begin{proof}
That $\tev_V$ and $\tcoev_V$ define a left duality is clear. The $\Uq$-linearity of $\ev^{(s)}_V$ and $\coev^{(s)}_V$ is verified directly. For example, we find
\[
\ev^{(s)}_V(E \cdot( v \otimes f))
=
(1-q^{2s})(-1)^{(\p v + \p 1) \p f} f(K^{-s}Ev)
\]
which vanishes, as required, if $q^{2s}=1$. The remaining statements are verified by similar calculations.
\end{proof}

The pivotal isomorphism $p^{(s)}: \id_{\cat} \Rightarrow (-)^{\vee \vee}$ corresponding to Lemma \ref{lem:catqPivot} has components $p^{(s)}_V = K^{1-s} \textnormal{ev}_V$, where $\textnormal{ev}_V: V \rightarrow V^{\vee \vee}$ is the canonical evaluation isomorphism of underlying super vector spaces. It is a consequence of Lemma \ref{lem:catqPivot} that projective and injective objects of $\cat$ coincide; see \cite[Proposition 6.1.3]{etingof2015}.

\subsubsection{Simple modules}
\label{sec:simpleMod}

Let $\Uqb$ be the subalgebra of $\Uq$ generated by $H, K^{\pm 1}$ and $E$. For each $\lambda \in \C$ and $\p p \in \Ztwo$, let $\mathbb{C}_{(\lambda,\p p)}$ be the one-dimensional super vector space with parity $\p p$ and $\Uqb$-module structure
\[
H \cdot 1 = \lambda,
\qquad
K \cdot 1 = q^{\lambda},
\qquad
E \cdot 1 = 0.
\]

\begin{Def}
The \emph{Verma module of highest weight $(\lambda+\red{r}-1, \p p)$} is
\[
V_{(\lambda,\p p)} = \Uq \otimes_{\Uqb} \C_{(\lambda+\red{r}-1,\p p)}.
\]
\end{Def}

The module $V_{(\lambda,\p p)}$ has a basis $\{v_n=F^n \cdot 1 \otimes 1 \mid n=0, \dots, \red{r}-1\}$ with $v_n$ of weight\footnote{For convenience, we often view the parity of a homogeneous weight vector as part of its weight.} $(\lambda+\red{r}-1-2n,\p p + \p n) \in \C \times \Ztwo$. Comparison of weights and the observation that the dual of weight module is a weight module yields an isomorphism
\begin{equation}
\label{eq:dualVerma}
V_{(\lambda,\p p)}^{\vee} \simeq V_{(-\lambda,\p p + \red{\p r} + \p 1)}.
\end{equation}
With respect to the pivotal structure of Lemma \ref{lem:catqPivot}, the quantum dimension is
\[
\textnormal{qdim}^{(s)}\, V_{(\lambda, \p p)}
=
\sum_{i=0}^{\red{r}-1} (-1)^{\p p + i} q^{(1-s)(\lambda+\red{r}-1-2i)} 
=
(-1)^{\p p} q^{(1-s)(\lambda+\red{r}-1)} \dfrac{1-(-1)^{\red{r}}q^{-2\red{r}}}{1+q^{-2}}
=0,
\]
the final equality following from equation \eqref{eq:q2r}.

For later use, define the \emph{Verma module of lowest weight $(\lambda-\red{r}+1, \p p)$} by
\[
V^{-}_{(\lambda, \p p)} = \Uq \otimes_{\Uqbn} \C^{-}_{(\lambda-\red{r}+1,\p p)},
\]
where $\Uqbn$ is the subalgebra generated by $H$, $K^{\pm 1}$ and $F$ and $\C^{-}_{(\lambda,\p p)}$ is the $\Uqbn$-module defined as above, but with the relation $F \cdot 1 = 0$ replacing $E \cdot 1 = 0$. The module $V_{(\lambda, \p p)}$ is a lowest weight Verma module if and only if it is simple (see Lemma \ref{lem:atypWeights} below), in which case a comparison of weights gives $V_{(\lambda, \p p)} \simeq V^{-}_{(\lambda, \p p+\overline{r}_0 + \p{1})}$.

\begin{Rem}
\label{rem:middleWeight}
We emphasize that $\lambda$ is the average weight of $V_{(\lambda,\p p)}$ and is the weight of a vector in $V_{(\lambda,\p p)}$ if and only if $\red{r} \equiv 1 \mod 2$ if and only if $r \equiv 4 \mod 8$.
\end{Rem}

Recall that $p_n \in \{0,1\}$ denotes the parity of $n \in \Z$.

\begin{Lem}
\label{lem:atypWeights}
The module $V_{(\lambda,\p p)}$ is simple if and only if
\[
\lambda \in \C \setminus \{\frac{r}{4} (2m+p_n+1) -\red{r} + n \mid m \in \Z, \; 1 \leq n \leq \red{r}-1\}.
\]
\end{Lem}

\begin{proof}
We search for singular vectors of $V_{(\lambda,\p p)}$. Let $1 \leq n \leq \red{r}-1$. Lemma \ref{lem:EFPow} gives
\[
E v_n
=
\pser{n} \pser{q^{\lambda + \red{r}-1}; -n+1} v_{n-1}.
\]
The restriction $1 \leq n \leq \red{r}-1$ ensures that $\pser{n} \neq 0$ so that $Ev_n=0$ if and only if $\pser{q^{\lambda + \red{r}-1}; -n+1} =0$. The latter equation has no solutions if and only if $\lambda$ is as in the statement of the lemma.
\end{proof}

Call $\lambda \in \C$ \emph{typical} if $V_{(\lambda,\p p)}$ is simple and \emph{atypical} otherwise. If $\lambda$ is atypical, then it can be written uniquely in the form
\begin{equation*}
\label{eq:atypWeight}
\lambda =\frac{r}{4} (2m+p_n+1) -\red{r} + n,
\qquad
m \in \Z, \; \; 1 \leq n \leq \red{r}-1.
\end{equation*}
For atypical $\lambda$, the unique non-trivial quotient $S_{(\lambda,\p p)}$ of $V_{(\lambda,\p p)}$ is simple of highest weight $\lambda + \red{r}-1$ and dimension $n$. The non-split short exact sequence
\begin{equation*}
\label{eq:filtVerm}
0 \rightarrow S_{(\lambda - 2n,\p p + \p n)} \rightarrow V_{(\lambda,\p p)} \rightarrow S_{(\lambda,\p p)} \rightarrow 0
\end{equation*}
implies that $S_{(\lambda,\p p)}$ is neither projective nor injective.

\begin{Thm}
\label{thm:simpleObjs}
A simple object of $\cat$ is isomorphic to $V_{(\lambda,\p p)}$ for $\lambda$ typical and $\p p \in \Ztwo$ or $S_{(\lambda,\p p)}$ for $\lambda$ atypical and $\p p \in \Ztwo$. 
\end{Thm}

\begin{proof}
This follows from the preceding discussion and the observation that any weight module has a homogeneous highest weight vector.
\end{proof}

\begin{Prop}
\label{prop:vermaProj}
If $\lambda \in \C$ is typical, then $V_{(\lambda, \p p)}$ is projective and injective. 
\end{Prop}

\begin{proof}
Let $f: W \twoheadrightarrow U$ be an epimorphism in $\cat$. A non-zero morphism $g: V_{(\lambda,\p p)} \rightarrow U$ is determined by a highest weight vector $g(v_0)=u_0 \in U$ of weight $(\lambda+\red{r}-1,\p p)$. Let $w_0 \in W$ be a preimage of $u_0$ of the same weight which satisfies $Ew_0 \in \ker f$. For all constants $c_0=1,c_1, \dots, c_{\red{r}-1} \in \mathbb{C}$, the weight vector
\[
w_0^{\prime} = \sum_{i=0}^{\red{r}-1} c_i F^i E^i w_0
\]
satisfies $f(w_0^{\prime})=u_0$. Using Lemma \ref{lem:EFPow}, we find that $E w_0^{\prime}=0$ if and only if
%
%
\[
\pser{i+1} \pser{q^{\lambda+\red{r}-1+2i+2}; -i} c_{i+1}
=
(-1)^{i+1} c_i
\qquad
i=0, \dots, \red{r}-2.
\]
Since $\lambda$ is typical, the coefficient of $c_{i+1}$ is non-zero and the system has a unique solution. This gives the desired lift of $g$ as $V_{(\lambda,\p p)} \rightarrow W$, $v_0 \mapsto w_0^{\prime}$.
\end{proof}

Let $\p p_1$, $\p p_2 \in \Ztwo$ and $\lambda_1, \lambda_2 \in \C$ such that $\lambda_1$, $\lambda_2$ and $\lambda_1 + \lambda_2$ are typical. Proposition \ref{prop:vermaProj} and consideration of characters yields
\begin{equation}
\label{eqn:tensorSimpGen}
V_{(\lambda_1, \p p_1)} \otimes V_{(\lambda_2, \p p_2)}
\simeq
\bigoplus_{k \in H_r}
V_{(\lambda_1+\lambda_2 + \red{r}-1, \p p_1 + \p p_2) +k},
\end{equation}
where $H_r=\{(-2 i,\p i) \mid 0 \leq i \leq \red{r}-1\}$. See \cite[Theorem 5.2]{costantino2015} for a similar computation.

\subsection{Braiding}
\label{sec:braiding}

Motivated by the universal $R$-matrix of the $\hbar$-adic quantum group of $\osp$ \cite{kulish1989,saleur1990}, in this section we construct a braiding on the category $\cat$.

Let $V, W \in \cat$. Define $\Upsilon_{V,W} \in \End_{\C}(V \otimes W)$ by
\[
\Upsilon_{V,W} (v \otimes w) = q^{\frac{1}{2}\lambda_v \lambda_w} v \otimes w,
\]
where $v\in V$ and $w \in W$ are of weight $\lambda_v$ and $\lambda_w$, respectively. Let $\tilde{R}_{V,W} \in \End_{\C}(V \otimes W)$ be the action (incorporating Koszul signs) of
\begin{equation}
\label{eq:RMatrix}
\tilde{R}
=
\sum_{l=0}^{\red{r}-1} (-1)^l \frac{q^{\frac{l(l-1)}{2}} (q-q^{-1})^{l}}{\pser{l}!}E^l \otimes F^l \in \Uq^{\otimes 2}.
\end{equation}
Finally, let $c_{V,W} =\tau_{V,W} \circ \Upsilon_{V,W} \circ \tilde{R}_{V,W}$, where $\tau$ is the Koszul symmetry on the monoidal category of super vector spaces.
Explicitly, for $v\in V$ and $w\in W$ we have
\begin{equation}\label{eqn:tau}
    \tau_{V,W}(v\otimes w)=(-1)^{\p v \p w} w\otimes v.
\end{equation}

\begin{Prop}
\label{prop:braiding}
The maps $\{c_{V,W} : V \otimes W \rightarrow W \otimes V\}_{V,W \in \cat}$ define a braiding on $\cat$.
\end{Prop}

\begin{proof}
Using Lemma \ref{lem:vanSum}, the inverse of $\tilde{R}$ is seen to be
\begin{equation*}
    \tilde{R}^{-1} = \sum_{l=0}^{\red{r}-1} (-q)^{- \frac{l(l-1)}{2}} \frac{(q-q^{-1})^l}{\pser{l}!} E^l \otimes F^l.
\end{equation*}
It follows easily from this that $c_{V,W}$ is a $\C$-linear isomorphism. Naturality of $c$ is clear.

Next, we prove $\Uq$-linearity of $c_{V,W}$. Let $v \in V$ and $w \in W$ be homogeneous of weight $\lambda_v$ and $\lambda_w$, respectively. Since $E^l \otimes F^l$ has $H$-weight zero, the action of $H$ commutes with each term in $\tilde{R}_{V,W}$. It follows that $c_{V,W}$ is $H$-linear. We verify $E$-linearity of $c_{V,W}$; verification of $F$-linearity is analogous. Setting $A_l = \frac{q^{\frac{l(l-1)}{2}} (q-q^{-1})^{l}}{\pser{l}!}$, we find that $E \cdot c_{V,W}(v \otimes w)$ is equal to
\[
\sum_{l=0}^{\red{r}-1} (-1)^{\p w (\p v + l)} A_l q^{\frac{1}{2}(\lambda_v + 2l) (\lambda_w - 2l)} ((-1)^{\p w + l} F^l w \otimes E^{l+1} v + q^{\lambda_v + 2l} EF^l w \otimes E^l v ).
\]
The signs appearing above are due to the graded structure \eqref{eqn:tau}. By applying Lemma \ref{lem:EFPow}, we can rewrite this as
\begin{align*}
    E \cdot c_{V,W}(v \otimes w)
    =&
    \sum_{l=0}^{\red{r}-1} (-1)^{\p w (\p v + l + 1) + l} A_l q^{\frac{1}{2}(\lambda_v + 2l) (\lambda_w - 2l)} F^l w \otimes E^{l+1} v \\ &+
    \sum_{l=0}^{\red{r}-1} (-1)^{\p w (\p v + l) + l} A_l q^{\frac{1}{2}(\lambda_v + 2l) (\lambda_w - 2l) + \lambda_v + 2l} F^lE w \otimes E^l v  \\
    &+
    \sum_{l=1}^{\red{r}-1} (-1)^{\p w (\p v + l)} A_l q^{\frac{1}{2}(\lambda_v + 2l) (\lambda_w - 2l) + \lambda_v + 2l} \pser{l}F^{l-1}\pser{K;-l+1}w \otimes E^l v.
\end{align*}
On the other hand, we compute
\begin{multline*}
c_{V,W}(E \cdot v \otimes w)
=
\sum_{l=0}^{\red{r}-1} (-1)^{l + l \p w + \p v \p w } A_l q^{\frac{1}{2}(\lambda_v + 2l)(\lambda_w - 2l + 2)} F^l Ew \otimes E^l v \\ + 
\sum_{l=0}^{\red{r}-1} (-1)^{l \p w + \p v \p w + \p w} A_l q^{\lambda_w + \frac{1}{2}(\lambda_v + 2l + 2)(\lambda_w - 2l)} F^l w \otimes E^{l+1}v.
\end{multline*}
The coefficients of $F^l Ew \otimes E^l v$ in the previous two expressions are both equal to
$
(-1)^{l + \p w (\p v + l)} A_l q^{\frac{1}{2}(\lambda_v + 2l)(\lambda_w - 2l + 2)}
$
whereas the coefficients of $F^l w \otimes E^{l+1}v$ are
\[
(-1)^{\p w (\p v + l + 1)} A_l q^{\lambda_w + \frac{1}{2}(\lambda_v + 2l + 2)(\lambda_w - 2l)}
\]
and
\begin{align*}
    &(-1)^{\p w(\p v + l + 1) + l} A_l q^{\frac{1}{2}(\lambda_v + 2l)(\lambda_w - 2l)}\\
    &+ 
    (-1)^{\p w(\p v + l + 1)} A_{l+1} q^{\frac{1}{2}(\lambda_v + 2l + 2)(\lambda_w - 2l - 2) + \lambda_v + 2l + 2} \pser{l+1} \frac{q^{\lambda_w - l} - (-1)^l q^{-\lambda_w + l}}{q - q^{-1}},
\end{align*}
respectively, which are seen to be equal.
This proves $E$-linearity of $c_{V,W}$.

Verification of the hexagon axioms is similar and so is omitted.
\end{proof}

\subsection{Ribbon structure}
\label{sec:ribbon}

Let $\Gr = \C \slash 2 \Z$. Define a grading $\cat = \bigoplus_{\p \lambda \in \Gr} \cat_{\p \lambda}$, where $\cat_{\p \lambda} \subset \cat$ is the full subcategory of modules whose weights are congruent to $\p \lambda$. Let $\SSS \subset \Gr$ be the image of the set $\{\lambda \in \C \mid \lambda - \red{r} +1 \mbox{ is atypical}\}$.

\begin{Prop}
\label{prop:genSS}
The $\Gr$-graded category $\cat$ is generically semisimple with small symmetric subset $\SSS$. For $\p \lambda \in \Gr \setminus \SSS$, a completely reduced dominating set of $\cat_{\p \lambda}$ is
\[
\{V_{(\lambda-\red{r}+1, \p p)} \mid \lambda \in \C \mbox{ is a lift of } \p \lambda, \, \p p \in \Ztwo\}.
\]
\end{Prop}

\begin{proof}
The set $\SSS$ is readily verified to be small and symmetric. Let $\p \lambda \in \Gr \setminus \SSS$ and $V \in \cat_{\p \lambda}$ non-zero. Since $\p \lambda \in \Gr \setminus \SSS$, a homogeneous highest weight vector of $V$ generates a submodule isomorphic to $V_{(\lambda-\red{r}+1, \p p)}$ for some lift $\lambda \in \C$ of $\p \lambda$ and $\p p \in \Ztwo$. Injectivity of typical Verma modules (Proposition \ref{prop:vermaProj}) ensures the existence of a splitting $V\simeq V^{\prime} \oplus V_{(\lambda-\red{r}+1, \p p)}$ with $V^{\prime} \in \cat_{\p \lambda}$ of dimension strictly less than $V$. Iterating this argument shows that $\cat_{\p \lambda}$ is semisimple with the claimed completely reduced dominating set.
\end{proof}

Given $V, V^{\prime} \in \cat$, consider the open Hopf link invariant
\[
\Phi_{V^{\prime},V} = (\id_{V} \otimes \ev_{V^{\prime}})\circ (c_{V^{\prime},V}\otimes \id_{V^{\prime \vee}})\circ (c_{V,V^{\prime}}\otimes \id_{V^{\prime \vee}})\circ (\id_V\otimes \tcoev_{V^{\prime}})\in \End_{\cat}(V).
\]
When $\End_{\cat}(V) \simeq \C$, write $\langle \Phi_{V^{\prime},V} \rangle \in \C$ for the scalar by which $\Phi_{V^{\prime},V}$ acts.

\begin{Lem}
\label{lem:Phis}
With respect to the pivotal structure of Lemma \ref{lem:catqPivot} determined by $s \in \Z$, there is an equality
\[
\langle \Phi^{(s)}_{V_{(\lambda^{\prime}, \p p^{\prime})},V_{(\lambda, \p p)}} \rangle
=
(-1)^{\p p^{\prime}} q^{(\lambda + \red{r} - s) (\lambda^{\prime}+\red{r}-1) +(1-\red{r}) \lambda}\cdot
\begin{cases}
\red{r} & \textnormal{if } \lambda \in \frac{r}{4}(2 \Z + \red{r}+1), \\
\frac{q^{\red{r} \lambda}-q^{-\red{r} \lambda}}{q^{\lambda}+(-1)^{\red{r}} q^{-\lambda}} & \textnormal{otherwise}.
\end{cases}
\]
\end{Lem}

\begin{proof}
Since $\End_{\cat}(V_{(\lambda, \p p)}) =\C \cdot \id_{V_{(\lambda, \p p)}}$, it suffices to compute the image of the highest weight vector under $\Phi^{(s)}_{V_{(\lambda^{\prime}, \p p^{\prime})},V_{(\lambda, \p p)}}$, to which only the swap and diagonal part $\Upsilon$ of the braiding contribute. We find
\[
\langle \Phi^{(s)}_{V_{(\lambda^{\prime}, \p p^{\prime})},V_{(\lambda, \p p)}} \rangle
=
(-1)^{\p p^{\prime}} q^{(\lambda + \red{r} - s) (\lambda^{\prime}+\red{r}-1)} \sum_{i=0}^{\red{r}-1} (-(-1)^{\red{r}}q^{-2\lambda})^i,
\]
and the claimed equality follows.
\end{proof}

\begin{Prop}
\label{prop:ribbonCat}
Assume that $r \not\equiv 4 \mod 8$ and give $\cat$ the pivotal structure of Lemma \ref{lem:catqPivot} determined by $s \in \Z$. The natural automorphism
\[
\theta : \id_{\cat} \Rightarrow \id_{\cat},
\qquad
\{\theta_V = \ptr_R(c_{V,V})\}_{V \in \cat}
\]
gives $\cat$ the structure of a $\C$-linear ribbon category if and only if $s = \red{r}$.
\end{Prop}

\begin{proof}
It is automatic that $\theta$ satisfies the balancing conditions of a ribbon structure. By generic semisimplicity and \cite[Theorem 2]{geer2018}, to prove that $\theta$ is a ribbon structure it suffices to verify that $\theta_{V^{\vee}} = \theta_{V}^{\vee}$ for all typical Verma modules.
The twist of a Verma module is determined by its value on the highest weight vector, to which only the swap and diagonal part $\Upsilon$ of the braiding contribute. We find
\[
\theta_{V_{(\lambda, \p p)}} = q^{\frac{(\lambda+\red{r}-1)(\lambda + \red{r}-2s+1)}{2}} \id_{V_{(\lambda, \p p)}}.
\]
%

Using the isomorphism \eqref{eq:dualVerma}, we see that $\theta_{V^{\vee}_{(\lambda, \p p)}} = \theta_{V_{(\lambda, \p p)}}^{\vee}$ if and only if $q^{2\lambda(\red{r}-s)} =1$. This equality holds for all typical $\lambda$ if and only if $s = \red{r}$. Since we assume that $r\not\equiv 4\mod8$, equation \eqref{eq:q2r} ensures that we can indeed set $s=\red{r}$ in Lemma \ref{lem:catqPivot}.
\end{proof}

We henceforth write $\ev$ and $\coev$ for $\ev^{(\red{r})}$ and $\coev^{(\red{r})}$.

\begin{Ex}
\label{ex:oneDimMod}
Using the relation $[E,F] = \frac{K -K^{-1}}{q-q^{-1}}$, the weight $\lambda \in \C$ of a one-dimensional $\Uq$-module is seen to satisfy $q^{2\lambda}=1$, equivalently, $\lambda \in \frac{r}{2} \Z$. Given $(k, \p p) \in \Z \times \Ztwo$, let $\C^H_{(\frac{kr}{2},\p p)}$ be the one-dimensional module of weight $\frac{kr}{2}$ and parity $\p p$. 
In the notation of Theorem \ref{thm:simpleObjs}, we have $\C^H_{(\frac{kr}{2},\p p)} = S_{(\frac{kr}{2}-\red{r}+1,\p p)}$.
Computing as in the proof of Proposition \ref{prop:ribbonCat}, we find
\begin{equation}
\label{eq:twistSigma}
\theta_{\C^H_{(\frac{kr}{2},\p p)}} = q^{\frac{1}{2}(\frac{kr}{2})^2 + (1-\red{r})\frac{kr}{2}}\id_{\C^H_{(\frac{kr}{2},\p p)}}. \qedhere
\end{equation}
\end{Ex}

\begin{Rem}
\label{rem:noRibbon}
Using the ``standard" form of the $R$-matrix---as in Section \ref{sec:braiding}---it is claimed in \cite[Theorem 4.1.3]{blumen2006} that a version of the small quantum group of $\mathfrak{osp}(1|2n)$ is a ribbon Hopf superalgebra for any root of unity of order at least $3$ . However, explicit details in the computation of a ribbon element are omitted. Proposition \ref{prop:ribbonCat} shows that when $r \equiv 4 \mod 8$, the ``standard" braiding and pivotal structure---as in Section \ref{sec:weightMod}---do not induce a ribbon structure on the restricted quantum group. Note also that, by Proposition \ref{prop:HopfNoGo}, there is no other restricted unrolled quantum group associated to $\osp$ which might resolve this issue. In future work, we give a Hopf algebraic perspective on this problem. Similar conditions on the order of $q$ were found in \cite{chenyang2018} for the Drinfeld double of a finite dimensional Taft superalgebra to have a ribbon element. For other examples of quantum groups which have braidings without compatible ribbon structures, see \cite[\S 10.4]{negron2021} and \cite[\S 10]{negron2023}.
\end{Rem}

\subsection{Non-degenerate relative pre-modularity}

\begin{Lem}[{\emph{cf}. \cite[Theorem 3.31]{anghel2021}}]
\label{lem:catUnimod}
The category $\cat$ is unimodular.
\end{Lem}

\begin{proof}
Since projectivity and injectivity in $\cat$ coincide, it suffices to prove self-duality of the injective hull of the trivial module $\unit=\C$. Let $\lambda \in \C$ be typical, so that $V_{(\lambda, \p 0)}$ is projective (Proposition \ref{prop:vermaProj}). Write $V_{(\lambda, \p 0)} \otimes V_{(\lambda, \p 0)}^{\vee} \simeq \bigoplus_{i=0}^n P_i$ as a direct sum of projective indecomposables. Adjunction and simplicity of $V_{(\lambda, \p 0)}$ give
\[
\Hom_{\cat}(\unit, V_{(\lambda, \p 0)} \otimes V_{(\lambda, \p 0)}^{\vee})
\simeq 
\Hom_{\cat}(V_{(\lambda, \p 0)}, V_{(\lambda, \p 0)})
\simeq \C
\]
so that the injective hull of $\unit$ appears in $V_{(\lambda, \p 0)} \otimes V_{(\lambda, \p 0)}^{\vee}$ with multiplicity one; call the corresponding summand $P_0$. Since $v_+:=v_0 \otimes v_{\red{r}-1}^{\vee}$ and $v_-:=v_{\red{r}-1} \otimes v_0^{\vee}$ span the weight spaces of $V_{(\lambda, \p 0)} \otimes V_{(\lambda, \p 0)}^{\vee}$ of weights $2(\red{r}-1)$ and $-2(\red{r}-1)$, respectively, we have $v_+ \in P_i$ and $v_- \in P_j$ for some $i$ and $j$ which satisfy $P_i^{\vee} \simeq P_j$. On the other hand, using Lemma \ref{lem:EFPow} we verify that $F^{\red{r}-1} v_+$ and $E^{\red{r}-1}v_-$ are non-zero $\Uq$-invariant vectors and so are elements of $P_0$. It follows that $i=j=0$ and $P_0$ is self-dual.
\comm{The last statement also uses the commutation relation $[F,E^{\red{r}-1}]$ and that $q^{2\red{r}}=(-1)^{\red{r}}$.}
\end{proof}

\begin{Prop}\label{prop:mTrace}
Up to a global scalar, there exists a unique modified trace $\mt$ on the ideal of projective objects of $\cat$.
\end{Prop}

\begin{proof}
The category $\cat$ is a locally finite pivotal $\C$-linear tensor category with enough projectives. Since $\cat$ is unimodular, \cite[Corollary 5.6]{geer2022} applies and we conclude that the ideal of projectives has a unique non-trivial right modified trace. Since $\cat$ is braided (Proposition \ref{prop:braiding}), this right modified trace is a modified trace. 
\end{proof}

Let $\lambda, \lambda^{\prime} \in \C$ be typical. Cyclicity of the modified trace implies
\[
\mt_{V_{(\lambda,\p p)} }(\Phi_{V_{(\lambda^{\prime}, \p p^{\prime})},V_{(\lambda,\p p)} })=\mt_{V_{(\lambda^{\prime}, \p p^{\prime})}}(\Phi_{V_{(\lambda,\p p)} ,V_{(\lambda^{\prime}, \p p^{\prime})}}).
\]
Evaluating each side of this equation using Lemma \ref{lem:Phis} and equation \eqref{eq:q2r} gives
\[
(-1)^{\p p^{\prime}} q^{\lambda\lambda^{\prime}}\frac{q^{\red{r} \lambda}-q^{-\red{r} \lambda}}{q^{\lambda}+q^{-\lambda}} \qd(V_{(\lambda,\p p)})
=
(-1)^{\p p} q^{\lambda\lambda^{\prime}}\frac{q^{\red{r} \lambda^{\prime}}-q^{-\red{r} \lambda^{\prime}}}{q^{\lambda^{\prime}}+q^{-\lambda^{\prime}}} \qd(V_{(\lambda^{\prime}, \p p^{\prime})}).
\]
We may therefore normalize the modified trace so that
\begin{equation}
\label{eq:qdimnormalized}
\qd(V_{(\lambda, \p p)}) =
(-1)^{\p p} \frac{q^{\lambda}+q^{-\lambda}}{q^{\red{r}\lambda}-q^{-\red{r}\lambda}}.
\end{equation}
With this normalization, we have
\begin{equation}
\label{eq:openHopfSimp}
\langle \Phi_{V_{(\lambda^{\prime}, \p p^{\prime})},V_{(\lambda, \p p)}} \rangle
=
(-1)^{\p p + \p p^{\prime}} \frac{q^{\lambda^{\prime} \lambda}}{\qd(V_{(\lambda, \p p)})}.
\end{equation}

Define
\[
\even = \gcd(2,r),
\qquad
\half = \frac{2}{\even},
\qquad
\parity{r}=\frac{r}{\even}.
\]
Let
\[
I_r =
\begin{cases}
\{-\red{r}+1+2 i \mid 0 \leq i \leq \parity{r}-1\} & \mbox{if } r \equiv 1 \mod 2 \mbox{ or } r \equiv 2 \mod 4, \\
\{-\red{r}+1+i \mid 0 \leq i \leq r-1\} & \mbox{if } r \equiv 0 \mod 8.
\end{cases}
\]
For $\p \lambda \in \Gr \setminus \SSS$ with chosen lift $\lambda \in \C$, set $\lambda+ I_r = \{\lambda +i \mid i \in I_r\}$.

\begin{Prop}
\label{prop:redMod}
Assume that $r \equiv 1 \mod 2$ or $r \equiv 2 \mod 4$. Give $\cat$ the generic semisimple structure of Proposition \ref{prop:genSS} and ribbon structure of Proposition \ref{prop:ribbonCat}. Let $\FR =\Z \times \Ztwo$. The monoidal functor $\sigma: \FR \rightarrow \cat_{\p 0}$, $(k,\p p) \mapsto \C^H_{(\half k r, \p p)}$,
is a free realization  and gives $\cat$ the structure of a non-degenerate pre-modular $\Gr$-category relative to $(\FR,\SSS)$.
\end{Prop}

\begin{proof}
Note that $\C^H_{(\half k r, \p p)}$ has weight $\half k r$ and so, since $\half r$ is even, lies in $\Gr$-degree $\p 0$. Equation \eqref{eq:twistSigma} gives $\theta_{\C^H_{(\half k r, \p p)}} = \id_{\C^H_{(\half k r, \p p)}}$. Let $\p \lambda \in \Gr$ with lift $\lambda \in \C$. We compute for the required bicharacter $\psi(\p \lambda, (k, \p p)) = q^{\half k r \lambda}$. If $\p \lambda \in \Gr \setminus \SSS$, then, in view of the completely reduced dominating set of Proposition \ref{prop:genSS}, we can take $\Theta(\p \lambda) = \{V_{(j,\p 0)} \mid j \in \lambda + I_r\}$.

For signs $a, b \in \{\pm \}$, define the generalized quadratic Gauss sum $G_{a,b} = \sum_{i=0}^{\parity{r}-1} q^{a2 i + b2 i^2}$. Let $\lambda \in \C$ be typical. Using that endomorphisms a Verma module consist of scalars, we find for the stabilization coefficients
\[
\Delta_b
=
\sum_{k \in I_r} \qd(V_{(\lambda+k,\p 0)}) \langle \theta^b_{V_{(\lambda-\red{r}+1,\p 0)}} \rangle \langle \theta^b_{V^{\vee}_{(\lambda+k,\p 0)}} \rangle \langle \Phi_{V^{-b}_{(\lambda+k,\p 0)},V_{(\lambda-\red{r}+1,\p 0)}} \rangle,
\]
where we have used the notation $V^+=V$ and $V^- = V^{\vee}$. Using equation \eqref{eq:q2r}, we compute
\[
\Delta_b
=
-b q^{-b(\red{r}-1)^2} \frac{q^{\lambda+1}G_{+,b} + q^{-\lambda-1}G_{-,b}}{q^{\lambda+1}+q^{-\lambda-1}}
=
-b q^{-b (\red{r}-1)^2} G_b,
\]
the second equality following from the observation that $G_{a,b}$ is independent of $a$; its common value is denoted by $G_b$. Evaluating the Gauss sum, we find
\[
\Delta_+
=
\sqrt{\parity{r}} q^{-\frac{3}{2}} \cdot
\begin{cases}
1 & \mbox{if } r \equiv 1 \mod 8, \\
-\I & \mbox{if } r \equiv 2 \mod 8, \\
-\I & \mbox{if } r \equiv 3 \mod 8, \\
-1 & \mbox{if } r \equiv 5 \mod 8, \\
-1 & \mbox{if } r \equiv 6 \mod 8, \\
\I & \mbox{if } r \equiv 7 \mod 8
\end{cases}
\]
and $\Delta_- = -\overline{\Delta_+}$.
\end{proof}

When $r \equiv 0 \mod 8$, direct modifications of  the computations above show that, with the $\half=1$ free realization, $\cat$ is relative pre-modular but degenerate: $\Delta_{\pm}=0$. This motivates the following modification of Proposition \ref{prop:redMod}. In particular, note that the grading group $\Gr$ is modified.

\begin{Prop}
\label{prop:redModZeroModEight}
Assume that $r \equiv 0 \mod 8$. Give $\cat$ the ribbon structure of Proposition \ref{prop:ribbonCat}. Let $\Gr= \C \slash \Z$ with subset $\SSS = \Z \slash \Z$ and $\cat = \bigoplus_{\p \lambda \in \Gr} \cat_{\p \lambda}$ the grading by $H$-weight modulo $\Z$. Let $\FR =\Z \times \Ztwo$. The monoidal functor $\sigma: \FR \rightarrow \cat_{\p 0}$, $(k,\p p) \mapsto \C^H_{(kr, \p p)}$, is a free realization and gives $\cat$ the structure of a non-degenerate pre-modular $\Gr$-category relative to $(\FR,\SSS)$.
\end{Prop}

\begin{proof}
Repeating the proof of Proposition \ref{prop:genSS} shows that $\cat$ is generically semisimple with small symmetric subset $\SSS$. The required bicharacter is $\psi(\p \lambda,(k,\p p))=q^{k r \lambda}$. A completely reduced dominating set of $\cat_{\p \lambda}$, $\p \lambda \in \Gr \setminus \SSS$, is
\[
\{V_{(\lambda-\red{r}+1, \p p)} \mid \lambda \in \C \mbox{ is a lift of } \p \lambda, \, \p p \in \Ztwo\}
\]
so that $\Theta(\p \lambda) = \{V_{(j, \p 0)} \mid j \in \lambda + I_r\}$. As in Proposition \ref{prop:redMod}, we find $\Delta_b = -b q^{-b (\red{r}-1)^2} G_b$, where $G_b= \sum_{i=0}^{r-1} (-1)^i q^{i + b \frac{i^2}{2}}$. Evaluating this expression gives $\Delta_+ = -\sqrt{r}q^{-\frac{3}{2}}e^{-\frac{3 \pi \I}{4}}$ and $\Delta_- = -\overline{\Delta_+}$.
\end{proof}

\subsection{Relative modularity}

Let $W \in \cat_{\p 0}$. Recall that a morphism $f \in \End_{\cat}(W)$ is \emph{transparent in $\cat_{\p 0}$} if 
\[
\id_U \otimes f = c_{W,U} \circ (f \otimes \id_U) \circ c_{U,W}
\]
and
\[
f \otimes \id_V =  c_{V,W} \circ (\id_V \otimes f) \circ c_{W,V}
\]
for all $U,V \in \cat_{\p 0}$.

\begin{Lem}[{\emph{cf}. \cite[Lemma 2.3]{derenzi2020}, \cite[Lemma 4.4]{garnerGeerYoung2025}}]
\label{lem:transparent}
Let $W \in \cat_{\p 0}$. A transparent morphism $f \in \End_{\cat}(W)$ factors through a finite direct sum of one-dimensional modules in $\cat_{\p 0}$.
\end{Lem}

\begin{proof}
Let $v_0$ be a highest weight vector of $V_{(-\red{r}+1,\p 0)} \in \cat_{\p 0}$ and $w \in W$ a weight vector. Because $Ev_0=0$, the explicit form of the braiding implies that $c_{V_{(-\red{r}+1,\p 0)},W}(v_0 \otimes w)$ is proportional to $w \otimes v_0$. Because $f$ is transparent, $c_{W,V_{(-\red{r}+1,\p 0)}}(f(w) \otimes v_0)$ is proportional to $v_0 \otimes f(w)$. Since $\{F^i v_0 \mid 0 \leq i \leq \red{r}-1\}$ is a basis of $V_{(-\red{r}+1,\p 0)}$, we conclude that $E f(w) =0$. Arguing in the same way with $V_{(-\red{r}+1,\p 0)}$ and $v_0$ replaced with $V^{-}_{(-\red{r}+1,\p 0)}$ and its lowest weight vector $v_0^-$, we conclude that $F f(w) =0$. It follows that $\frac{K-K^{-1}}{q-q^{-1}} = [E,F]$ annihilates $f(w)$. Writing $\lambda \in \C$ for the weight of $f(w)$, we conclude that $q^{2\lambda} =1$ so that, by Example \ref{ex:oneDimMod}, each homogeneous weight vector in the image of $f$ spans a one-dimensional module. Since $W$ is in degree $\p 0$ and the image of $f$ is a direct sum of its weight spaces, we conclude that $f$ factors through a direct sum of one-dimensional modules of degree $\p 0$.
\end{proof}

Recall the notion of relative modular category from Definition \ref{def:modG}.

\begin{Thm}
\label{thm:relMod}
Assume that $r \not\equiv 4 \mod 8$. With the structures of Propositions \ref{prop:redMod} and \ref{prop:redModZeroModEight}, the category $\cat$ is $\Gr$-modular relative to $(\SSS,\FR)$ with relative modularity parameter
\[
\zeta
=
\begin{cases}
-\parity{r} & \mbox{if }r \equiv 1 \mod 2 \mbox{ or } r \equiv 2 \mod 4,\\
-r & \mbox{if }r \equiv 0 \mod 8.
\end{cases}
\]
\end{Thm}

\begin{proof}
It remains to establish the existence of a relative modularity parameter. Consider Definition \ref{def:modG} with $h = \p \gamma$ and $g = \p \lambda$ with $V_i=V_{(\alpha,\p 0)},V_j= V_{(\beta,\p 0)} \in \Theta(\p \lambda)$. Denote by $f_{\p \gamma; \alpha, \beta} \in \End_{\cat}(V_{(\alpha,\p 0)} \otimes V_{(\beta,\p 0)}^{\vee})$ the morphism determined by the left hand side of diagram \eqref{eq:mod}. The handle slide property of \cite[Lemma 5.9]{costantino2015} guarantees that $f_{\p \gamma; \alpha, \beta}$ is transparent in $\cat_{\p 0}$. By Lemma \ref{lem:transparent}, we can write
\[
f_{\p \gamma; \alpha, \beta}
=
\sum_{i=1}^m h_{\p \gamma; \alpha, \beta, i} \circ g_{\p \gamma; \alpha, \beta, i}
\]
for some
\[
g_{\p \gamma; \alpha, \beta, i} \in \Hom_{\cat}(V_{(\alpha,\p 0)} \otimes V_{(\beta,\p 0)}^{\vee},\C^H_{(\frac{k_i r}{2},\p p_i)})
\]
and
\[
h_{\p \gamma; \alpha, \beta, i} \in \Hom_{\cat}(\C^H_{(\frac{k_i r}{2},\p p_i)},V_{(\alpha,\p 0)} \otimes V_{(\beta,\p 0)}^{\vee})
\]
with $\C^H_{(\frac{k_i r}{2},\p p_i)} \in \cat_{\p 0}$; see Example \ref{ex:oneDimMod}.

Assume first that $r \equiv 1 \mod 2$ or $r \equiv 2 \mod 4$ so that $\FR$ contains all one-dimensional modules in degree $\p 0$. The set $I_{\p \lambda}$ is defined so that
\[
\Hom_{\cat}(V_{(\alpha,\p 0)} \otimes V^{\vee}_{(\beta,\p 0)}, \C^H_{(\half k_i r, \p p_i)})
\simeq
\Hom_{\cat}(V_{(\alpha,\p 0)}, V_{(\beta,\p 0)} \otimes \C^H_{(\half k_i r, \p p_i)})
\]
is zero unless $\alpha= \beta$ and $(k_i,\p p_i) = (0,\p 0)$. We may therefore assume that $\alpha=\beta$. In this case, $f_{\p \gamma; \alpha, \alpha}$ factors through the trivial module and
\begin{equation}
\label{eq:relModLarge}
f_{\p \gamma; \alpha, \alpha}
=
\zeta \tcoev_{V_{(\alpha,\p 0)}} \circ \ev_{V_{(\alpha,\p 0)}}
\end{equation}
for some constant $\zeta$. Applying $\mt_{V_{(\alpha,\p 0)} \otimes V_{(\alpha,\p 0)}^{\vee}}$ to the right hand side of equation \eqref{eq:relModLarge} gives $\zeta \qd(V_{(\alpha,\p 0)})$ while applying it to the left hand side gives
\begin{eqnarray*}
\mt_{V_{(\alpha,\p 0)} \otimes V_{(\alpha,\p 0)}^{\vee}} (f_{\p \gamma; \alpha, \alpha})
&=&
\sum_{\delta \in \gamma+ I_r} \qd(V_{(\delta,\p 0)}) \mt_{V_{(\alpha,\p 0)} \otimes V_{(\alpha,\p 0)}^{\vee}}(f_{\delta; \alpha, \alpha})\\
&=&
\qd(V_{(\alpha,\p 0)}) \sum_{\delta \in \gamma+ I_r}
\mt_{V_{(\delta,\p 0)}} \big(\Phi_{V^{\vee}_{(\alpha,\p 0)},V_{(\delta,\p 0)}} \big) \mt_{V_{(\delta,\p 0)}} \big( \Phi_{V_{(\alpha,\p 0)},V_{(\delta,\p 0)}} \big)\\
&=&
\qd(V_{(\alpha,\p 0)}) \sum_{\delta \in \gamma+ I_r} \qd(V_{(\delta,\p 0)})^2 \langle \Phi_{V^{\vee}_{(\alpha,\p 0)},V_{(\delta,\p 0)}} \rangle \langle \Phi_{V_{(\alpha,\p 0)},V_{(\delta,\p 0)}} \rangle \\
&=&
-\qd(V_{(\alpha,\p 0)}) \parity{r}.
\end{eqnarray*}
The second equality follows from isotopy invariance, defining properties of $\mt$ and simplicity of $V_{(\delta,\p 0)}$. The final equality follows from the isomorphism \eqref{eq:dualVerma} and equation \eqref{eq:openHopfSimp}. In particular, the sign in the final expression results from equation \eqref{eq:openHopfSimp} and the fact that the highest weight vector of $V^{\vee}_{(\alpha,\p 0)}$ is of odd parity, since $\red{r}$ is even. We conclude that $\zeta = -\parity{r}$.

If instead $r \equiv 0 \mod 8$, then
\[
\Hom_{\cat}(V_{(\alpha,\p 0)} \otimes V^{\vee}_{(\beta,\p 0)}, \C^H_{(\frac{k_i r}{2},\p p_i)})
\simeq
\Hom_{\cat}(V_{(\alpha,\p 0)}, V_{(\beta,\p 0)} \otimes \C^H_{(\frac{k_i r}{2},\p p_i)})
\]
is non-zero for a unique $\C^H_{(\frac{k_i r}{2},\p p_i)}$ with $k_i \in \{0,1\}$ and $\p p_i = \p 0$; call it $\C^H_{(\frac{k r}{2},\p 0)}$. It follows that 
\begin{equation}
\label{eq:relModLarge8}
f_{\p \gamma; \alpha, \beta}
=
d_{\alpha,\beta} \id_{\C^H_{(\frac{k r}{2},\p 0)}} \otimes \tcoev_{V_{(\alpha,\p 0)}} \circ \ev_{V_{(\alpha,\p 0)}}
\end{equation}
for some $d_{\alpha,\beta} \in \C$. The modified trace of the right hand side of equation \eqref{eq:relModLarge8} is $\pm d_{\alpha,\beta} \qd(V_{(\alpha,\p 0)})$ and so vanishes precisely when $d_{\alpha,\beta}=0$. Computing as in the previous paragraph, we find the modified trace of the left hand side of equation \eqref{eq:relModLarge8} to be
\[
\sum_{\delta \in \gamma+ I_r} \qd(V_{(\delta,\p 0)}) \mt_{V_{(\alpha+k,\p 0)} \otimes V_{(\alpha,\p 0)}^{\vee}}(f_{\delta; \alpha+k,\alpha})
=
-\qd(V_{(\alpha,\p 0)}) q^{k \gamma} \sum_{i=0}^{r-1} q^{\frac{kr}{2} i}.
\]
When $k=0$ the sum over $i$ is equal to $r$. When $k=1$ we have $q^{\frac{r}{2}}=-1$ and the sum over $i$ vanishes. The constant $d_{\alpha,\beta}$ therefore vanishes unless $\alpha = \beta$, in which case it is equal to $-r$. It follows that $\zeta = -r$.
\end{proof}

\section{Topological field theory from \texorpdfstring{$\Uq$}{Uq(osp(1|2))}}
\label{sec:tft}

In this section, we use the representation theoretic results of Section \ref{sec:unrolledOsp}, specifically Theorem \ref{thm:relMod},  to construct a family of decorated three-dimensional TFTs. In particular, in this section we assume that $r \not\equiv 4 \mod 8$. All manifolds are assumed oriented.

\subsection{Non-semisimple TFTs from relative modular categories}

Fix a relative modular category $\cat$. Let $\Cob_{\cat}$ be the category of decorated surfaces and their diffeomorphism classes of admissible decorated bordisms, as defined in \cite[\S 2]{derenzi2022}. An object of $\Cob_{\cat}$ is a tuple $\CS=(\Sigma, \{x_i\}, \coh, \mathcal{L})$ consisting of
\begin{itemize}
\item a closed surface $\Sigma$ with a choice $*$ of basepoint for each connected component,
\item a finite set $\{x_i\} \subset \Sigma \setminus *$ of oriented framed $\cat$-coloured points,
\item a cohomology class $\coh \in H^1(\Sigma\setminus\{x_i\}, * ;\Gr)$ such that $\coh(\mer_i) = g_i$ is the degree of the colour of $x_i$, where $\mer_i$ is the oriented boundary of a regular neighbourhood of $x_i$, and
\item a Lagrangian subspace ${\mathcal L}\subset H_1(\Sigma; \R)$.
\end{itemize}
A morphism in $\Cob_{\cat}$ is a tuple $\mathcal{M} = (M,T,\coh,m) : \CS_1 \rightarrow \CS_2$ consisting of
\begin{itemize}
\item a bordism $M: \Sigma_1 \rightarrow \Sigma_2$,
\item a $\cat$-coloured ribbon graph $T \subset M$ whose colouring is compatible with those of the marked points of $\CS_1$ and $\CS_2$,
\item a cohomology class $\coh \in H^1(M\setminus T, *_1 \cup *_2; \Gr)$ which restricts to $\coh_j$ on $\Sigma_j$, $j=1,2$, and such that the colour of each connected component $T_c$ of $T$ has degree $\coh(\mer_c) \in \Gr$, where $\mer_c$ is an oriented meridian of $T_c$, and
\item an integer $m \in \Z$, the \emph{signature defect}.
\end{itemize}
Moreover, it is required that $\mathcal{M}$ is \emph{admissible}: for each connected component $M_c$ of $M$ which is disjoint from $\Sigma_1$, at least one edge of $T \cap M_c$ is coloured by a projective object of $\cat$ or there exists an embedded closed oriented curve $\gamma \subset M_c$ such that $\coh(\gamma) \in \Gr \setminus \SSS$. Disjoint union gives $\Cob_{\cat}$ the structure of a symmetric monoidal category.

The unique pairing $\llangle -, - \rrangle: \FR \times \FR \rightarrow \{\pm 1\}$ which makes the diagram
\[
\begin{tikzpicture}[baseline= (a).base]
\node[scale=1.0] (a) at (0,0){
\begin{tikzcd}[column sep=7.0em,row sep=2.0em]
\sigma_{k_1} \otimes \sigma_{k_2} \arrow{r}[above]{c_{\sigma_{k_1},\sigma_{k_2}}} \arrow{d}[left]{\wr} & \sigma_{k_2} \otimes \sigma_{k_1} \arrow{d}[right]{\wr} \\
\sigma_{k_1+k_2} \arrow{r}[below]{\llangle k_1,k_2 \rrangle \cdot \id_{\sigma_{k_1+k_2}}} & \sigma_{k_2+k_1}
\end{tikzcd}
};
\end{tikzpicture}
\]
commute for all $k_1, k_2 \in \FR$ induces a symmetric braiding on the monoidal category $\ZVect_{\C}$ of $\FR$-graded vector spaces and their degree preserving linear maps.

\begin{Thm}[{\cite[Theorem 6.2]{derenzi2022}}]
\label{thm:relModTQFT}
A modular $\Gr$-category $\cat$ relative to $(\FR,\SSS)$ together with a choice $\mathcal{D}$ of square root of the relative modularity parameter $\zeta$ defines a symmetric monoidal functor
\[
\TQFT_{\cat}: \Cob_{\cat} \rightarrow \ZVect_{\C}.
\]
\end{Thm}

We refer to $\TQFT_{\cat}$ as the TFT associated to $\cat$. The values of $\TQFT_{\cat}$ on closed bordisms $\varnothing \rightarrow \varnothing$ coincide with a (renormalization of the) $3$-manifold invariants $N_{\cat}$ of \cite{costantino2014}; see Definition \ref{def:cgpInvtOrig} below. The former values are defined as follows. Let $R$ be a $\cat$-coloured ribbon graph in $S^3$. Suppose that an edge of $R$ is coloured by a generic simple object $V \in \cat$. Let $R_V$ be the $(1,1)$-ribbon graph obtained from $R$ by cutting an edge labelled by $V$. Then $F^{\prime}_{\cat}(R) :=\mt_V(R_V) \in \C$ is an isotopy invariant of $R$ \cite[Theorem 3]{geer2009}. With this notation, the partition function of a closed admissible bordism $\mathcal{M}=(M,T,\coh,m)$ with computable surgery presentation $L \subset S^3$ is
\begin{equation}
\label{eq:cgpInvt}
\TQFT_{\cat}(\mathcal{M})
=
\D^{-1-l} \big( \frac{\D}{\Delta_{-}} \big)^{m-\sigma(L)} F'_{\cat}(L\cup T) \in \C.
\end{equation}
Here $L$ has $l$ connected components, $\sigma(L)$ is the signature of the linking matrix of $L$ and each component $L_c$ of $L$ is coloured by the Kirby colour $\Omega_{\coh(\mer_c)}$.

We note for later use that the construction of $\TQFT_{\cat}$ requires the choice of an element $g_0 \in \Gr$ and a simple projective $V_{g_0} \in \cat_{g_0}$. Up to equivalence, $\TQFT_{\cat}$ is independent of these choices.

In the following sections we take $\cat$ to be one of the relative modular categories of Theorem \ref{thm:relMod}. The pairing $\llangle -, - \rrangle: \FR \times \FR \rightarrow \{\pm 1\}$ is
\[
\llangle (k_1,\p p_1),(k_2,\p p_2) \rrangle = (-1)^{\p p_1 \p p_2}.
\]
The category $\cat$ is TFT finite in the sense of \cite[Definition 1.15]{geerYoung2025}, as follows easily from the observation that $\cat$ is a Krull--Schmidt category whose isomorphism classes of indecomposable projectives are in bijection with isomorphism classes of simple objects. It follows from \cite[Theorem 1.16]{geerYoung2025} that all state spaces of $\TQFT_{\cat}$ are finite dimensional.

\subsection{Verlinde formulae} 
\label{sec:verlinde}

Let $\cat$ be one of relative modular categories constructed in Theorem \ref{thm:relMod}. In this section we compute the value of $\TQFT:=\TQFT_{\cat}$ on trivial circle bundles over surfaces. Together with a combinatorial description of spanning sets of state spaces of $\TQFT$, this allows for the computation of Euler characteristics and total dimensions of state spaces. These results can be seen as Verlinde formulae for $\TQFT$.

Let\footnote{We omit the required Lagrangian subspace of $H_1(\Sigma;\R)$ since it is not used in what follows. Similarly, we omit the signature defect required to define morphisms in $\Cob_{\cat}$.} $\CS=(\Sigma,\{x_1,\ldots, x_s\},\coh)$ be a decorated admissible connected surface of genus $g \geq 1$. For each $\p \lambda \in \Gr$, define the decorated $3$-manifold
\[
\CS \times S^1_{\p \lambda}=(\Sigma \times S^1,T=\{x_1, \dots, x_s\}\times S^1,\coh \oplus \p \lambda),
\]
where
\[
\coh \oplus \p \lambda
\in
H^1(\Sigma \setminus \{x_i\}; \Gr) \oplus \Gr
\simeq H^1((\Sigma \times S^1) \setminus T; \Gr).
\]
The partition function of $\CS \times S^1_{\p \lambda}$ can be computed via an explicit surgery presentation, as in \cite[\S 5.3]{blanchet2016}, \cite[\S 3.2]{geerYoung2025}. Writing $V_{(\mu_i-\red{r}+1,\p p_i)}$ for the colour of $x_i$ and setting $\mu = \sum_{i=1}^s (\mu_i -\red{r}+1)$ and $\p p = \sum_{i=1}^s \p p_i$, we find
\begin{equation}
\label{eq:verlinde}
\TQFT(\CS \times S^1_{\p \lambda})
=
(-1)^{\p p} \zeta^{g-1}\sum_{k \in I_r} q^{\mu (\lambda+k)} \left(\frac{q^{\red{r}(\lambda+k)}-q^{-\red{r}(\lambda-k)}}{q^{\lambda+k}+q^{-\lambda-k}} \right)^{2g-2-s}.
\end{equation}

For $(k,\p p) \in \FR$, let $\TQFT_{(k,\p p)}(\CS) \subset \TQFT(\CS)$ be the subspace of $\FR$-degree $(k,\p p)$. Define the generating function of $\FR$-graded dimensions of $\TQFT(\CS)$ to be
\[
\dim_{(x,y)} \TQFT(\CS) = \sum_{(k, \p p) \in \FR} (-1)^{\p p} \dim_{\C} \TQFT_{(k,\p p)}(\CS) x^k y^{\p p} \in \Z[x^{\pm 1},y^{\pm 1}].
\]
Here we treat the parity subgroup $\Ztwo$ as the multiplicative group $\{1, -1\}$.

\begin{Thm}[Verlinde formula]
\label{thm:verlinde}
There is an equality
\[
\TQFT(\CS \times S^1_{\p \lambda}) = \dim_{(q^{\half r \p \lambda},1)} \TQFT(\CS).
\]
\end{Thm}

\begin{proof}
The equality can be proved by a direct modification of the proofs of \cite[Theorem 5.9]{blanchet2016} and \cite[Theorem 3.5]{geerYoung2025}. The only change is the explicit form of the bicharacter $\psi$, given in the present setting in Propositions \ref{prop:redMod} and \ref{prop:redModZeroModEight}, which determines the appropriate specialization of the variables $x$ and $y$.
\end{proof}

\begin{Cor}
\label{cor:VerlindeEuler}
When $\CS$ has no marked points, that is, $s=0$, the Euler characteristic of $\TQFT(\CS)$ with respect to the parity subgroup is
\[
\chi(\TQFT(\CS))
=
\begin{cases}
\parity{r} & \mbox{if } r \equiv 1 \mod 2, \\
\parity{r} & \mbox{if } r \equiv 2 \mod 4, \\
r & \mbox{if } r \equiv 0 \mod 8.
\end{cases}
\]
if $g =1$ and
\[
\chi(\TQFT(\CS))
=
\begin{cases}
0 & \mbox{if } r \equiv 1 \mod 2, \\
0 & \mbox{if } r \equiv 2 \mod 4, \\
\frac{r^{3g-3}}{2^{2g-3}}& \mbox{if } r \equiv 0 \mod 8
\end{cases}
\]
if $g \geq 2$.
\end{Cor}

\begin{proof}
Theorem \ref{thm:verlinde} gives
\[
\chi(\TQFT(\CS)) = \lim_{\p \lambda \rightarrow 0} \TQFT(\CS \times S^1_{\p \lambda}).
\]
Since $q^{\red{r}}$ is a sign (see equation \eqref{eq:q2r}), when $s=0$ equation \eqref{eq:verlinde} becomes
\begin{equation}
\label{eq:verlindeNoPoints}
\TQFT(\CS \times S^1_{\p \lambda})
=
\zeta^{g-1}\sum_{k \in I_r} \left(\frac{q^{\red{r}\lambda}-q^{-\red{r}\lambda}}{q^{\lambda+k}+q^{-\lambda-k}} \right)^{2g-2}.
\end{equation}
When $g=1$, we obtain
\[
\chi(\TQFT(\CS))
=
\vert I_r \vert
=
\begin{cases}
\parity{r} & \mbox{if } r \equiv 1 \mod 2 \mbox{ or } r \equiv 2 \mod 4, \\
r & \mbox{if } r \equiv 0 \mod 8.
\end{cases}
\]
Consider then the case $g \geq 2$. If $r \equiv 1 \mod 2$ or $r \equiv 2 \mod 4$, then
\begin{equation}
\label{eq:limForEuler}
\lim_{\lambda \rightarrow 0} \frac{q^{\red{r}\lambda}-q^{-\red{r}\lambda}}{q^{\lambda+k}+q^{-\lambda-k}}
\end{equation}
vanishes for all $k \in I_r$. Indeed, the limit of the numerator vanishes while that of the denominator does not. It follows that $\chi(\TQFT(\CS))=0$. If instead $r \equiv 0 \mod 8$, then $\lim_{\lambda \rightarrow 0} (q^{\lambda+k}+q^{-\lambda-k})$ is zero if and only if $k \in \{\pm \frac{r}{2}\}$, in which case an application of l'H\^{o}pital's rule gives for the limit \eqref{eq:limForEuler} the value $\I \red{r}$ and we arrive at $\chi(\TQFT(\CS)) = \frac{r^{3g-3}}{2^{2g-3}}$.
\end{proof}

We turn to the computation of the total dimension $\dim_{\C} \TQFT(\CS)$. In isolation, Theorem \ref{thm:verlinde} is insufficient to do so. Indeed, without constraints on the $\FR$-support of $\TQFT(\CS)$, the total dimension cannot be accessed as a limit of $\TQFT(\CS \times S^1_{\p \lambda})$. A similar issue was encountered and resolved in \cite[\S 3.2]{geerYoung2025} in the context of TFTs constructed from the unrolled quantum group of $\gloo$. We explain the modifications for $\osp$.

Consider the following oriented trivalent graph $\widetilde{\Gamma}$:
\begin{equation*}\label{eqn:Gamma'Triv}
\begin{tikzpicture}[anchorbase] 
\draw[thick,decoration={markings, mark=at position 0.275 with {\arrow{>}},mark=at position 0.775 with {\arrow{<}}},postaction={decorate}] (4.5,0) circle (0.35);
\draw[thick,decoration={markings, mark=at position 0.275 with {\arrow{>}},mark=at position 0.775 with {\arrow{<}}},postaction={decorate}] (6.5,0) circle (0.35);
\node at (1.5,-2)  {\tiny $e_1$};
\node at (0,-0.6)  {\tiny $e_2$};
\node at (2,-0.6)  {\tiny $e_3$};

\node at (0,0.6)  {\tiny $f_1$};
\node at (1,0.3)  {\tiny $f_2$};
\node at (2,0.6)  {\tiny $f_3$};

\node at (3.5,0)  {$\cdots$};

\node at (4.5,0.6)  {\tiny $f_{2g-5}$};
\node at (5.5,0.3)  {\tiny $f_{2g-4}$};
\node at (6.5,0.6)  {\tiny $f_{2g-3}$};
\node at (4.5,-0.6)  {\tiny $e_{g-1}$};
\node at (6.5,-0.6)  {\tiny $e_g$};

\node at (7.75,1.0)  {\tiny $e^{-1}_{-1}$};
\node at (8.5,1.0)  {\tiny $e^{0}_{-1}$};
\node at (9.25,1.0)  {\tiny $e^1_{-1}$};
\node at (7.25,-0.2)  {\tiny $e_{g+1}$};
\node at (8.1,-0.3)  {\tiny $e_{g+1}^{\prime}$};
\node at (8.85,-0.25)  {\tiny $e_1^{\prime}$};
\node at (7.5,-2.0)  {\tiny $e_1$};

\draw[-,thick,decoration={markings, mark=at position 0.15 with {\arrow{<}},mark=at position 0.35 with {\arrow{<}},mark=at position 0.55 with {\arrow{<}}},postaction={decorate}] (6.85,0) to [out=0,in=90] (10.0,-0.95);
\draw[-,thick,decoration={markings, mark=at position 0.5 with {\arrow{>}}},postaction={decorate}] (6.85,-1.7) to [out=0,in=-90] (10.0,-0.95);
\draw[-,thick,decoration={markings, mark=at position 0.5 with {\arrow{<}}},postaction={decorate}] (7.75,0.75) to (7.75,0.025);
\draw[-,thick,decoration={markings, mark=at position 0.5 with {\arrow{<}}},postaction={decorate}] (8.5,0.75) to (8.5,0.05) ;
\draw[-,thick,decoration={markings, mark=at position 0.5 with {\arrow{<}}},postaction={decorate}] (9.25,-0.05) to (9.25,0.75);
\draw[-,thick,decoration={markings, mark=at position 0.5 with {\arrow{>}}},postaction={decorate}] (6.15,0) to (4.85,0);

\draw[thick,decoration={markings, mark=at position 0.275 with {\arrow{>}},mark=at position 0.775 with {\arrow{<}}},postaction={decorate}] (0,0) circle (0.35);
\draw[thick,decoration={markings, mark=at position 0.275 with {\arrow{>}},mark=at position 0.775 with {\arrow{<}}},postaction={decorate}] (2,0) circle (0.35);
       
\draw[-,thick,decoration={markings, mark=at position 0.5 with {\arrow{>}}},postaction={decorate}] (1.65,0) to (0.35,0);
\draw[-,thick,decoration={markings, mark=at position 0.5 with {\arrow{>}}},postaction={decorate}] (-0.35,0) to [out=220,in=180] (2,-1.7);

\node at (3.5,0)  {$\cdots$};
\node at (6.5,-1.7)  {$\dots$};
\end{tikzpicture}.
\end{equation*}

\begin{Def}
\label{def:colouring}
A \emph{colouring of $\widetilde{\Gamma}$ of degree $k \in \FR$} is a function $\Co: \textnormal{Edge}(\widetilde{\Gamma}) \rightarrow \textnormal{Ob}(\cat)$ such that \[
\Co(e_{-1}^1) = V_{g_0} = \Co(e_{-1}^{-1}),
\qquad
\Co(e_{-1}^0)=\sigma_k
\]
and assigning to each of the remaining edges $e$ a simple module $V_{(\alpha_e-\red{r}+1,\p p_e)}$, where $(\alpha_e,\p p_e) \in \C \times \Ztwo$ is congruent to $\coh(\mer_e) \in \Gr$, such that the following \emph{Balancing Condition} holds: at each trivalent vertex, the algebraic sum of the colours of incident edges is consistent with the isomorphism \eqref{eqn:tensorSimpGen}.
\end{Def}

Fix a lift $\tilde{\coh} \in H^1(\Sigma; \C \times \Ztwo)$ of $\coh \in H^1(\Sigma; \Gr)$ and let
\[
\mathfrak{C}_k
=
\{\mbox{degree $k$ colourings of $\widetilde{\Gamma}$} \mid \Co(e_i) = V_{\tilde{\coh}(\mer_{e_i}) +j} \mbox{ for some } j \in I_r, \; i=1,\dots, g \}.
\]
Set $\mathfrak{C} = \bigsqcup_{k \in \FR} \mathfrak{C}_k$. As explained in \cite[\S 1.6]{geerYoung2025}, each colouring $c \in \mathfrak{C}_k$ defines a vector $v_c =(\tilde{\eta},\widetilde{\Gamma}_c,\coh,0) \in \TQFT_{-k}(\CS)$. The construction is based off the fact that $\widetilde{\Gamma}$ is a modification of an oriented spine of a genus $g$ handlebody with a open ball removed; the modification is required to construct vectors of non-zero degree.

\begin{Thm}
\label{thm:genusgStateSpace}
Let $\CS=(\Sigma, \coh)$ be a decorated connected surface of genus $g \geq 1$ without marked points. Assume that $2\coh$ is not in the image of the canonical map $H^1(\Sigma; \SSS) \rightarrow H^1(\Sigma; \Gr)$ and $\coh(\mer_{e_1})+g_0 \not\in \SSS$. Then $\{v_c \mid c \in \mathfrak{C}_k\}$ spans $\TQFT_{-k}(\CS)$. Moreover, $\TQFT_{-k}(\CS)$ is trivial unless $k=(d, \p d)$ for some $d \in \Z$ which, if $r \equiv 0 \mod 8$, must be even.
\end{Thm}

\begin{proof}
A direct generalization of the argument from the proof of \cite[Theorems 3.3 and 4.1]{geerYoung2025} establishes the spanning statement; this uses the stated assumptions on $\coh$.

Next, we study the set $\mathfrak{C}$. Assume that $g \geq 2$; the slightly degenerate case $g=1$, where $e_1=f_0$, is dealt with similarly. An element of $\mathfrak{C}$ can be constructed as follows. For each $i=1, \dots, g$, colour $e_i$ by a simple module $V_{(\alpha_i-\red{r}+1, \p p_i)}$ satisfying the constraint in the definition of $\mathfrak{C}_k$. Next, colour $f_1, f_2, \dots, f_{2g-3}$ recursively so as to satisfy the Balancing Conditions. The colours $\{V_{(\beta_i-\red{r}+1,\p q_i)}\}$ of $\{f_i\}$ are determined by $(\epsilon_1, \dots, \epsilon_{2g-3}) \in \{-i \mid 0 \leq i \leq \red{r}-1\}^{\times 2g-3}$ through the initial condition $\beta_1 = \alpha_1 - \alpha_2 -2\epsilon_1$ and the recursive system
\[
\begin{cases}
\beta_{2i} =\beta_{2i-1} +\alpha_{i+1} + 2\epsilon_{2i} & \mbox{for } 1 \leq i \leq g-2, \\
\beta_{2i+1} + \alpha_{i+2} + 2\epsilon_{2i+1}=\beta_{2i} & \mbox{for }1 \leq i \leq g-2.
\end{cases}
\]
The solution is
\[
\begin{cases}
\beta_{2i}=\alpha_1+ 2\sum_{j=1}^{2i} (-1)^j 2\epsilon_j,
 & \mbox{for } 1 \leq i \leq g-2, \\
\beta_{2i+1} = \alpha_1 - \alpha_{i+2} +2\sum_{j=1}^{2i+1} (-1)^j \epsilon_j & \mbox{for }1 \leq i \leq g-2.
\end{cases}
\]
Similarly, we find for the parities
\[
\begin{cases}
\p q_{2i}=\p p_1+ \sum_{j=1}^{2i} \p \epsilon_j & \mbox{for } 1 \leq i \leq g-2, \\
\p q_{2i+1} = \p p_1 + \p p_{i+2}+ \sum_{j=1}^{2i+1} \p \epsilon_j & \mbox{for } 1 \leq i \leq g-2.
\end{cases}
\]
Writing $\sigma_k=\C^H_{(\half k r, \p p)}$, the colours of the edges $e_{g+1}$, $e_{g+1}^{\prime}$ and $e_1^{\prime}$ are subject only to the Balancing Conditions associated to the four top right vertices of $\widetilde{\Gamma}$:
\[
(\alpha_1,\p p_1) + g_0 +2 \epsilon^{\prime} = (\alpha_1^{\prime},\p p_1^{\prime}) ,
\]
\[
(\alpha_{g+1}^{\prime}, \p p_{g+1}^{\prime}) + (\half k r, \p p) = (\alpha_1^{\prime},\p p_1^{\prime}),
\]
\[
(\alpha_{g+1}^{\prime}, \p p_{g+1}^{\prime})
=
(\alpha_{g+1}, \p p_{g+1}) + g_0 + 2 \epsilon^{\prime \prime},
\]
\[
(\alpha_{g+1},\p p_{g+1}) = (\alpha_{g},\p p_g) + (\beta_{2g-3},\p q_{2g-3}) + 2 \epsilon^{\prime \prime \prime}.
\]
Here $\epsilon^{\prime}, \epsilon^{\prime \prime}, \epsilon^{\prime \prime \prime} \in \{-i \mid 0 \leq i \leq \red{r}-1\}$. These equations hold if and only if
\[
\half k r=2\epsilon^{\prime} - 2 \epsilon^{\prime \prime} -2 \epsilon^{\prime \prime \prime} - 2\sum_{j=1}^{2g-3} (-1)^j \epsilon_j,
\qquad
\p p = \p \epsilon^{\prime} + \p \epsilon^{\prime \prime} + \p \epsilon^{\prime \prime \prime}+ \sum_{j=1}^{2g-3} \p \epsilon_j.
\]
and
\[
\alpha_{g+1}=\alpha_1 + 2\epsilon^{\prime \prime \prime}+2\sum_{j=1}^{2g-3} (-1)^j \epsilon_j,
\qquad
\p p_{g+1}=\p p_1+ \p \epsilon^{\prime \prime \prime} + \sum_{j=1}^{2g-3} \p \epsilon_j.
\]
Setting $d = \epsilon^{\prime} -  \epsilon^{\prime \prime} - \epsilon^{\prime \prime \prime} - \sum_{j=1}^{2g-3} (-1)^j \epsilon_j$ gives $(\half k r,\p p)=(2d, \p d)$. If $r \equiv 1 \mod 2$ or $r \equiv 2 \mod 4$, then $\frac{\half r}{2}$ is odd and we conclude that this colouring is of degree $(d, \p d) \in \FR$ for some $d \in \Z$. If $r \equiv 0 \mod 8$, then $\frac{r}{2}$ is even so that $d$ is even. It is immediate that the above construction recovers each element of $\mathfrak{C}$.
\end{proof}

\begin{Cor}
\label{cor:VerlindeTotalDim}
Assume that $g \geq 2$. When $s=0$, the total dimension of $\TQFT(\CS)$ is
\[
\dim_{\C} \TQFT(\CS)
=
r^{3g-3} \cdot \begin{cases}
2^{2g-2} & \mbox{if } r \equiv 1 \mod 2, \\
\frac{1}{2^{g-1}} & \mbox{if } r \equiv 2 \mod 4, \\
\frac{1}{2^{2g-3}} & \mbox{if } r \equiv 0 \mod 8.
\end{cases}
\]
\end{Cor}

\begin{proof}
When $r \equiv 0 \mod 8$, Theorem \ref{thm:genusgStateSpace} implies that $\TQFT(\CS)$ has even parity, whence $\dim_{\C} \TQFT(\CS)$ is equal to the Euler characteristic computed in Corollary \ref{cor:VerlindeEuler}.

When $r \equiv 1 \mod 2$ or $r \equiv 2 \mod 4$, Theorem \ref{thm:genusgStateSpace} implies that the generating function of graded dimensions simplifies to
\[
\dim_{(x,y)} \TQFT(\CS) = \sum_{d\in \Z} (-1)^{\p d} \dim_{\C} \TQFT_{(d,\p d)}(\CS) x^d y^{\p d}.
\]
Since $\lim_{\lambda \rightarrow \frac{1}{2 \half}} q^{\half \lambda r} = -1$, Theorem \ref{thm:verlinde} and equation \eqref{eq:verlindeNoPoints} lead to the equality
\begin{equation*}
\label{ex:VerlindeTotalDim}
\dim_{\C} \TQFT(\CS)
=
\lim_{\lambda \rightarrow \frac{1}{2 \half}} \zeta^{g-1}\sum_{k \in I_r} \left(\frac{q^{\red{r}\lambda}-q^{-\red{r}\lambda}}{q^{\lambda+k}+q^{-\lambda-k}} \right)^{2g-2}.
\end{equation*}
We have $\lim_{\lambda \rightarrow \frac{1}{2 \half}} (q^{\red{r}\lambda}-q^{-\red{r}\lambda}) = 0$. Writing $k = -\red{r}+1+2n$ and noting that $q^{\red{r}}=1$, we see that $\lim_{\lambda \rightarrow \frac{1}{2 \half}} (q^{\lambda+1+2n}+q^{-\lambda-1-2n}) =0$ if and only if $q^{4n+ 2+ \frac{1}{\half}}=-1$. The latter holds if and only if $4n+2+\frac{1}{\half} \equiv \frac{r}{2} \mod r$. The unique element $-\red{r}+1+2n_* \in I_r$ which satisfies the latter equation has
\[
n_*
=
\begin{cases}
\frac{5r-5}{8} & \mbox{if } r \equiv 1 \mod 8,\\
\frac{3r+2}{8} & \mbox{if } r \equiv 2 \mod 8,\\
\frac{7r-5}{8} & \mbox{if } r \equiv 3 \mod 8,\\
\frac{r-5}{8} & \mbox{if } r \equiv 5 \mod 8,\\
\frac{r-6}{8} & \mbox{if } r \equiv 6 \mod 8,\\
\frac{3r-5}{8} & \mbox{if } r \equiv 7 \mod 8.
\end{cases}
\]
An application of l'H\^{o}pital's rule now gives the claimed total dimension.
\end{proof}

\section{\texorpdfstring{Comparison between $N_r^{\osp}$ and $\Zhat^{\osp}$-invariants}{Comparison between N_r^{osp(1|2)} and Zhat^{osp(1|2)}-invariants}}
\label{sec:Zhat}

The goal of this section is to prove Theorems \ref{thm:sltwoOspZhatIntro} and \ref{thm:cgpVsZhatIntro} from the Introduction.

\subsection{Plumbed \texorpdfstring{$3$}{3}-manifolds}
\label{sec:plumbed3Mfld}

Let $\Gamma=(\Gamma,f)$ be a plumbing graph, that is, $\Gamma$ is a tree with (finite) sets $V(\Gamma)$ and $E(\Gamma)$ of vertices and edges, respectively, and vertex weighting $f: V(\Gamma) \rightarrow \Z$. We view $E(\Gamma)$ as a collection of unordered pairs of distinct elements of $V(\Gamma)$. Let $B$ be the $\vert V(\Gamma) \vert \times \vert V(\Gamma)\vert$-matrix with entries
\[
B_{v v^{\prime}}
=
\begin{cases}
1 & \mbox{if } \{v, v^{\prime}\} \in E(\Gamma), \\
f_{v} & \mbox{if } v=v^{\prime}, \\
0 & \mbox{else}.
\end{cases}
\]
We assume that $B$ is invertible. Let $b_+$ (resp. $b_-$) be the number of positive (resp. negative) eigenvalues of $B$ and $\sigma=b_+ - b_-$ the signature of $B$. Note that $b_+ + b_- = \vert V(\Gamma) \vert$. The number of edges incident to a vertex $v$ is its degree $\deg(v)$. The plumbing graph $(\Gamma,f)$ is called \emph{weakly negative definite} if the restriction of $B^{-1}$ to vertices of degree greater than two is negative definite \cite[\S 4.3]{gukov2021b}.

Let $L \subset S^3$ be the framed link with an unknot component for each vertex $v \in V(\Gamma)$ with framing $f_v$ and with distinct components $v$ and $v^{\prime}$ Hopf linked whenever $B_{v v^{\prime}}=1$. The linking matrix of $L$ is therefore $B$. Let $M$ be the closed oriented $3$-manifold obtained by integral surgery along $L$.

\subsection{\texorpdfstring{$\Zhat$-invariants for $\osp$}{Zhat-invariants for osp(1|2)}}

Let $\Gamma$ be a plumbing graph. Define formal power series
\begin{equation}
\label{eq:Fpm}
F^{\pm}(\{x_v\}_{v \in V(\Gamma)})
=
\prod_{v \in V(\Gamma)} \left( x_v \pm \frac{1}{x_v} \right)^{2 - \deg v} \in \C\pser{\{x_v, x_v^{-1}\}_{v \in V(\Gamma)}}.
\end{equation}

\begin{Lem}
\label{lem:FPlusVsMin}
There is an equality $F^-(\{\I x_v\}_{v \in V(\Gamma)}) = - F^+(\{x_v\}_{v \in V(\Gamma)})$.
\end{Lem}

\begin{proof}
We have
\[
F^-(\{\I x_v\}_{v \in V(\Gamma)})
=
\I^{2 \vert V(\Gamma) \vert - \sum_v \deg v} F^+(\{x_v\}_{v \in V(\Gamma)}).
\]
Now use that, since $\Gamma$ is a tree, $2 \vert V(\Gamma) \vert - \sum_v \deg v = 2$.
\end{proof}

Define complex numbers $\{F_l^{\pm}\}_{l \in \Z^{V(\Gamma)}}$ through the equality
\[
F^{\pm}(\{x_v\}_{v \in V(\Gamma)})
=
\sum_{l \in \Z^{V(\Gamma)}} F^{\pm}_l \prod_{v \in V(\Gamma)} x_v^{l_v}.
\]
It is immediate that $F^+_l = 0$ unless $l \equiv \deg v \mod 2$, and similarly for $F^-_l$. Moreover, Lemma \ref{lem:FPlusVsMin} gives
\begin{equation}
\label{eq:FpmCoeffs}
F^+_l
=
- \I^{l^t \colOne} F^-_l,
\qquad
l \in \Z^{V(\Gamma)}
.\end{equation}
where $\colOne$ denotes the column vector $(1, \dots, 1)\in \Z^{V(\Gamma)}$.

Let $\qq$ be an indeterminate.

\begin{Def}[{\cite[\S 4.2]{chauhan2023}}]
\label{def:ospZhat}
Let $\Gamma$ be a weakly negative definite plumbing graph with linking matrix $B$ and associated closed oriented $3$-manifold $M$. The \emph{$\Zhat^{\osp}$-invariant of $M$ with $\Spinc$ structure $\mathfrak{s}$} is
\begin{equation}
\label{eq:ospZhat}
\Zhat^{\osp}_{\mathfrak{s}} (M;\qq)= \qq^{\frac{3 \sigma- \tr B}{4}} \sum_{\substack{l \in \Z^{V(\Gamma)} \\ l \equiv \mathfrak{s} \mod 2 B \Z^{V(\Gamma)}}} F^+_l \qq^{- \frac{1}{4}l^t B^{-1} l} \in \qq^{\Delta^{\osp}_{\mathfrak{s}}} \Q \pser{\qq}[\qq^{-1}]
\end{equation}
for some scalar $\Delta^{\osp}_{\mathfrak{s}} \in \Q$.
\end{Def}

Keeping the notation of Definition \ref{def:ospZhat}, we recall from \cite[Appendix A]{gukov2017} that the $\Zhat$-invariant for $\sltwo$ is
\begin{equation}
\label{eq:sltwoZhat}
\Zhat^{\sltwo}_{\mathfrak{s}} (M;\qq)= (-1)^{b_+} \qq^{\frac{3 \sigma- \tr B}{4}} \sum_{\substack{l \in \Z^{V(\Gamma)} \\ l \equiv \mathfrak{s} \mod 2 B \Z^{V(\Gamma)}}} F^-_l \qq^{- \frac{1}{4}l^t B^{-1} l} \in \qq^{\Delta^{\sltwo}_{\mathfrak{s}}} \Q \pser{\qq}[\qq^{-1}]
\end{equation}
for some $\Delta^{\sltwo}_{\mathfrak{s}} \in \Q$.

\begin{Rem}
\begin{enumerate}
\item Definition \ref{def:ospZhat} generalizes the $\Zhat^{\osp}$-invariants of \cite[Eqn. (79)]{chauhan2023} from negative definite to weakly negative definite plumbing graphs. The additional overall factor of $2^{-\vert V(\Gamma) \vert}$ which appears in \cite[Eqn. (79)]{chauhan2023} is absorbed into the coefficients $F_l^-$ used in this paper; compare equations \eqref{eq:Fpm} and \cite[Eqn. (80)]{chauhan2023}. In contrast to $\Zhat^{\sltwo}$, the series $\Zhat^{\osp}$ does not have a global factor of $(-1)^{b_+}$. Quite generally, $\Zhat^{\mathfrak{g}}$ for a simple (say) Lie superalgebra is expected to have a global factor of $(-1)^{\vert R_+ \vert b_+}$, where $\vert R_+ \vert$ is the number of positive roots of $\mathfrak{g}$; see \cite[Def. 2.2]{park2020} and \cite[Eqn. (2.10)]{ferrari2024}. In particular, $\vert R_+ \vert$ is equal to $1$ and $2$ for $\sltwo$ and $\osp$, respectively, explaining the difference in global signs.

\item Weak negative definiteness of $\Gamma$ ensures convergence of the $\qq$-series $\Zhat^{\sltwo}_{\mathfrak{s}}(M; \qq)$ in the unit disk $\{\qq \mid \vert \qq \vert < 1 \}$ \cite{gukov2021b}. A direct modification of this argument applies to $\Zhat^{\osp}_{\mathfrak{s}}(M; \qq)$.

\item That the $\qq$-series \eqref{eq:ospZhat} is independent of the plumbing presentation of $M$, and hence an invariant of the pair $(M,\mathfrak{s})$, can be verified by a direct modification of the calculation for $\Zhat^{\sltwo}$ given in \cite[Proposition 4.6]{gukov2021b}. For example, the global factor $(-1)^{b_+}$ appearing in equation \eqref{eq:sltwoZhat} is required to cancel a sign which appears due to the oddness of the function $z -\frac{1}{z}$ (seen as a factor of $F^-$) under the substitution $z \mapsto z^{-1}$. In the case of $\osp$ we consider instead the function $z + \frac{1}{z}$ (seen as a factor of $F^+$), which is even under the same substitution, and hence do not need to cancel a factor of $(-1)^{b_+}$.
\end{enumerate}
\end{Rem}

\subsection{\texorpdfstring{$\Zhat^{\sltwo}$ versus $\Zhat^{\osp}$}{Zhat^{sl(2)} versus Zhat^{osp(1|2)}}}

The following result is motivated by physical calculations of Chauhan and Ramadevi \cite[\S\S 4-5]{chauhan2023}.

\begin{Thm}
\label{thm:sltwoOspZhat}
For each $\Spinc$ structure $\mathfrak{s}$ on $M$, there exists a root of unity $c_{\mathfrak{s}} \in \C$ whose order is divisible by $8 \vert H_1(M;\Z)\vert$ such that
\[
\Zhat^{\osp}_{\mathfrak{s}}(M; -\qq)
=
c_{\mathfrak{s}} \Zhat_{\mathfrak{s}}^{\sltwo}(M; \qq).
\]
\end{Thm}

\begin{proof}
In the notation of Section \ref{sec:surgeryTopology}, write $\mathfrak{s} = \sigma(b,s)$ for some $b \in H_1(M;\Z)$ and $s \in \Spin(M)$. Writing the summation index $l$ in the definition of $\Zhat^{\osp}_{\mathfrak{s}}$ as $2b+B(s-\colOne) + 2 Bk$ for $k \in \Z^{V(\Gamma)}$ and using equation \eqref{eq:FpmCoeffs} to replace $F_l^+$ with $F_l^-$, we find
\begin{multline*}
\Zhat^{\osp}_{\sigma(b,s)}(M;\qq)
=
- \qq^{\frac{3 \sigma- \tr B}{4} -\frac{(2b+B(s-\colOne))^t B^{-1} (2b+B(s-\colOne))}{4}} \I^{(2b+B(s-\colOne))^t \colOne} \\ \sum_{k \in \Z^{V(\Gamma)}} (-1)^{k^t B \colOne} F^-_{l(k)} \qq^{-k^t B k - k^t(2b+B(s-\colOne))}.
\end{multline*}
The definitions give
\[
k^t B k + k^t Bs
=
2\sum_{1 \leq i<j \leq \vert V(\Gamma) \vert} B_{ij} k_i k_j + \sum_{i=1}^{\vert V(\Gamma) \vert} B_{ii} k^2_i + \sum_{i=1}^{\vert V(\Gamma) \vert}  \big(\sum_{j=1}^{\vert V(\Gamma) \vert} B_{ij} s_j \big) k_i.
\]
Since $\sum_{j=1}^{\vert V(\Gamma) \vert} B_{ij} s_j \equiv B_{ii} \mod 2$, we have $k^tB k + k^t Bs \equiv 0 \mod 2$, whence
\[
-k^t B k - k^t(2b+B(s-\colOne))
\equiv
k^t B\colOne
\mod 2.
\]
We conclude that
\[
\sum_{k \in \Z^{V(\Gamma)}} (-1)^{k^t B \colOne} F^-_{l(k)} \qq^{-k^t B k - k^t(2b+B(s-\colOne))}
=
\sum_{k \in \Z^{V(\Gamma)}} F^-_{l(k)} (-\qq)^{-k^t B k - k^t(2b+B(s-\colOne))}.
\]
Hence, we have
\begin{multline*}
\Zhat^{\osp}_{\sigma(b,s)}(M;\qq)
=
- \qq^{\frac{3 \sigma- \tr B}{4} -\frac{(2b+B(s-\colOne))^t B^{-1} (2b+B(s-\colOne))}{4}} \I^{(2b+B(s-\colOne))^t \colOne} \\ \sum_{k \in \Z^{V(\Gamma)}} F^-_{l(k)} (-\qq)^{-k^t B k - k^t(2b+B(s-\colOne))}.
\end{multline*}
Making the same substitution as above for $l$ in $\Zhat^{\sltwo}_{\sigma(b,s)}(M;\qq)$, we see that we may take $c_{\sigma(b,s)} = e^{2 \pi \I \tilde{c}_{\sigma(b,s)}}$, where
\[
\tilde{c}_{\sigma(b,s)}
=
\frac{1+b_+}{2} + \frac{(2b+B(s-\colOne))^t \colOne}{4} + \frac{\delta}{2} \in \Q
\]
with
\[
\delta
=
\frac{3 \sigma- \tr B}{4} -\frac{(2b+B(s-\colOne))^t B^{-1} (2b+B(s-\colOne))}{4}.
\]
Explicitly, we have
\[
\tilde{c}_{\sigma(b,s)}
=
\frac{4(1+b_+) + 3 \sigma - \tr B -4 b^t (s - \colOne) - (s-\colOne)^t B (s- \colOne) }{8} - \frac{b^t B^{-1} b}{2}
\]
Since the numerator of the first term is an integer $b^t B^{-1} b \in \frac{1}{\det B} \Z$ and $\vert \det B \vert = \vert H_1(M;\Z) \vert$, we conclude that $\tilde{c}_{\sigma(b,s)} \in \frac{1}{8 \vert H_1(M;\Z) \vert}\Z$.
\end{proof}

\subsection{Relation to CGP invariants}

To compare $\Zhat$-invariants and CGP invariants, it is useful to slightly renormalize the latter, as defined by equation \eqref{eq:cgpInvt}. Fix a relative modular category $\cat$.

\begin{Def}[{\cite[\S 2.4]{costantino2014}}]
\label{def:cgpInvtOrig}
Let $L \subset S^3$ be an oriented framed link and $M$ the closed oriented $3$-manifold obtained by integral surgery along $L$. Let $\coh \in H^1(M; \Gr)$ such that $\coh(\mer_c) \in \Gr \setminus \SSS$ for each component $L_c$ of $L$. The \emph{CGP invariant} of $(M,\coh)$ is
\begin{equation}
\label{eq:cgpInvtOrig}
N_{\cat}(M,\coh)
=
\frac{1}{\Delta_+^{b_+} \Delta_-^{b_-}} F^{\prime}_{\cat}(L) \in \C,
\end{equation}
where each component $L_c$ of $L$ is coloured by the Kirby colour $\Omega_{\coh(\mer_c)}$ and $b_+$ (resp. $b_-$) is the number of positive (resp. negative) eigenvalues of the linking matrix of $L$.
\end{Def}

Take now $\cat$ to be one of the relative modular categories of Theorem \ref{thm:relMod}. We write $N_r^{\osp}(M,\coh)$ for $N_{\cat}(M,\coh)$ to emphasize the dependence on the order $r$ of $q$. Given $\lambda, \lambda^{\prime} \in \C$, set
\[
\qd(\lambda)
=
\frac{q^{\lambda}+q^{-\lambda}}{q^{\red{r}\lambda}-q^{-\red{r}\lambda}},
\qquad
S(\lambda^{\prime},\lambda) = q^{\lambda^{\prime} \lambda},
\qquad
T(\lambda)
=
q^{\frac{\lambda^2-(\red{r}-1)^2}{2}}.
\]
By Proposition \ref{prop:ribbonCat} and equations \eqref{eq:qdimnormalized} and \eqref{eq:openHopfSimp}, we have $\langle \theta_{V_{(\lambda,\p p)}} \rangle = T(\lambda)$ and, for $\lambda$ typical,
\[
\qd(V_{(\lambda, \p p)}) = (-1)^{\p p} \qd(\lambda),
\qquad
\langle \Phi_{V_{(\lambda^{\prime}, \p p^{\prime})},V_{(\lambda, \p p)}} \rangle
=
(-1)^{\p p^{\prime}} \frac{S(\lambda, \lambda^{\prime})}{\qd(\lambda)}.
\]

Let $\Gamma$ be a weakly negative definite plumbing graph and $\coh \in H^1(M; \Gr)$. Write $\alpha_v$ for the value of $\coh$ on the homology class of the oriented meridian of the $v$\textsuperscript{th} component of $\Gamma$. For $k \in I_r^{V(\Gamma)}$, write $\alpha_k = \{\alpha_v + k_v\}_{v \in V(\Gamma)}$. With this notation, the CGP invariant \eqref{eq:cgpInvtOrig} becomes
\begin{equation}
\label{eq:plumbingCGP}
N_r^{\osp}(M,\coh)
=
\frac{1}{\Delta_+^{b_+} \Delta_-^{b_-}} \sum_{k \in I_r^{V(\Gamma)}} \prod_{v \in V(\Gamma)} \qd(\alpha_{k_v})^{2 - \deg(v)} T(\alpha_{k_v})^{f_v} \prod_{\{v_1,v_2\} \in E(\Gamma)} S(\alpha_{k_{v_1}}, \alpha_{k_{v_2}}).
\end{equation}

In \cite[Theorem 4.18]{costantino2023}, a regularization of the function $\Zhat^{\sltwo}$ via an additional parameter $t$ is presented so that the limit $\qq \rightarrow e^{\frac{2\pi \I}{r}}$ of particular sums involving $\Zhat^{\sltwo}_{\mathfrak{s}} (M;\qq)$ may be computed as the evaluation at $\qq = e^{\frac{2\pi \I}{r}}$ followed by the limit at $t=1$. Similar techniques should be applied for more general Lie (super)algebras, like $\osp$, but we do not pursue this. Instead, we work in the following setting.

\begin{Hyp}
\label{hyp:techAssump}
We assume that a similar regularization exists for $\Zhat^{\osp}_{\mathfrak{s}} (M;\qq)$, so that the limit in Theorem \ref{thm:cgpVsZhat} below may be computed as the evaluation at $\qq = e^{\frac{4\pi \I}{r}}$ followed by the limit $t\to 1$. 
\end{Hyp}

The goal of the remainder of the paper is to prove the following comparison result. We use the topological notation of Section \ref{sec:surgeryTopology}.

\begin{Thm}
\label{thm:cgpVsZhat}
Let $\delta \in \{\pm 1\}$ and assume that $r$ is congruent modulo $8$ to $\delta$ or $2 \delta$. Let $\Gamma$ be a weakly negative definite plumbing graph. There is an equality
\[
N_r^{\osp}(M,\coh)
=
\lim_{\qq \rightarrow e^{\frac{4\pi \I}{r}}} \sum_{\mathfrak{s} \in \Spinc(M)} c^{\osp}_{\coh,\mathfrak{s}} \Zhat^{\osp}_{\mathfrak{s}} (M;\qq),
\]
where
\[
c^{\osp}_{\coh,\sigma(b,s)}
=
\frac{e^{\pi{\I} \mu(M,s)} \tor(M,[4 \coh])}{\vert H_1(M;\Z) \vert} \sum_{a,f} e^{2 \pi \I \left(-\frac{r-\delta}{8} \lk(a,a) - \lk(a,b+f) +2 \lk(f,f)-\frac{1}{2} \coh(a) \right)}
\]
if $r \equiv \delta \mod 8$ and
\[
c^{\osp}_{\coh,\sigma(b,s)}
=
\frac{e^{-\delta \frac{\pi \I}{2} \mu(M,s)} \tor(M,[2 \coh])}{\vert H_1(M;\Z) \vert} \sum_{a,f} e^{2 \pi \I \left(-\frac{r-2\delta}{8} \lk(a,a) - \lk(a,b+\delta f) + \delta \lk(f,f) - \frac{1}{2} \coh(a) \right)}
\]
if  $r \equiv 2 \delta \mod 8$.
Here $\sum_{a,f}$ denotes summation over $a,f \in H_1(M;\Z)$ and $\tor$ is the appropriately normalized Reidemeister torsion of $M$ (defined below).
\end{Thm}

Before proving Theorem \ref{thm:cgpVsZhat}, we record a number of comments.

\begin{Rem}
\label{rem:ZhatVsCGP}
\begin{enumerate}
\item The radial limit $\qq \rightarrow e^{\frac{4\pi \I}{r}}$ appearing in Theorem \ref{thm:cgpVsZhat} should be interpreted as in the case of $\sltwo$ \cite[Footnote 4]{costantino2023}. The subtlety is that the function $\Zhat^{\osp}_{\mathfrak{s}} (M;\qq)$ is multivalued because of the overall factor of $\qq^{\Delta^{\osp}_{\mathfrak{s}}}$.

\item From the perspective of Theorem \ref{thm:cgpVsZhat}, Hypothesis \ref{hyp:techAssump} allows the limit $\qq \rightarrow e^{\frac{4\pi \I}{r}}$ to be computed as an evaluation at $\qq = e^{\frac{4\pi \I}{r}}$. This is the form in which Theorem \ref{thm:cgpVsZhat} is proved below. The $\sltwo$-analogue of Hypothesis \ref{hyp:techAssump} is verified for Y-shaped plumbing graphs in \cite[Appendix D]{costantino2023}.

\item The analogue of Theorem \ref{thm:cgpVsZhat} for $\sltwo$ is \cite[Theorem 4.17]{costantino2023}, where the cases $r \equiv 2 \delta \mod 8$ and $r \equiv 4 \mod 8$ are treated. In the former case, the coefficients $c^{\sltwo}_{\coh,\sigma(b,s)}$ and $c^{\osp}_{\coh,\sigma(b,s)}$ are very closely related. Relabelling $f$ with $-f$, a summand of $\sum_{a,f}$ in $c^{\osp}_{\coh,\sigma(b,s)}$ becomes
\[
e^{2 \pi \I \left(-\frac{r-2\delta}{8} \lk(a,a) + \delta \lk(a,f - \delta b) - \frac{1}{2} \coh(a)+ \delta \lk(f,f)\right)},
\]
which is the corresponding summand in $c^{\sltwo}_{\coh,\sigma(b,s)}$. The overall prefactors differ by the replacement of $\tor(M,[2 \coh])$ in $c^{\osp}_{\coh,\sigma(b,s)}$ with $-\tor(M,[\coh])$ in $c^{\sltwo}_{\coh,\sigma(b,s)}$. It follows that we have
\[
-\tor(M,[\coh]) c^{\osp}_{\coh,\sigma(b,s)}
=
\tor(M,[2\coh]) c^{\sltwo}_{\coh,\sigma(b,s)}.
\]

\item The invariants $N_r^{\sltwo}$ when $r \equiv 1 \mod 2$ were recently defined by Detcherry \cite{detcherry2025}. Based on the calculations of Section \ref{sec:pm3Mod8}, where it is explained why the computation involved in the proof of Theorem \ref{thm:cgpVsZhat} does not extend to the case $r \equiv \pm 3 \mod 8$, we expect $N_r^{\sltwo}$ to be closely related to $\Zhat^{\sltwo}$ when $r \equiv \pm 1 \mod 8$ but not when $r \equiv \pm 3 \mod 8$.

\item When $r \equiv 0 \mod 8$ the computation from the proof of Theorem \ref{thm:cgpVsZhat} again fails, although for a different reason than for $r \equiv \pm 3 \mod 8$; see Section \ref{sec:0Mod8}. The reason is at least partially due to the $\C \slash \Z$-grading of $\cat$ of Proposition \ref{prop:redModZeroModEight}, as opposed to the $\C \slash 2 \Z$-grading used when $r \not\equiv 0 \mod 8$. The former grading leads to non-topological terms in the putative relation between $N_r^{\osp}$ and $\Zhat^{\osp}$. 

\item When $r \equiv 0 \mod 8$ the invariant $N_r^{\sltwo}$ is not defined since the category of $\Uqsltwo$-weight modules with its standard $\C \slash 2 \Z$-grading and free realization is degenerate. Instead, there exist spin-refined CGP invariants \cite{blanchet2014} which, as shown in \cite[\S 4.3]{costantino2023}, are related to $\Zhat^{\sltwo}$-invariants. The spin-refinement uses the same $\C \slash 2 \Z$-grading with a modified free realization. Motivated by this, we expect Theorem \ref{thm:cgpVsZhat} to extend to the case $r \equiv 0 \mod 8$ when $N_r^{\osp}$ is replaced with its (not yet defined) spin-refinement which results from the category $\cat$ with its natural $\C \slash 2 \Z$-grading, as opposed to the $\C \slash \Z$-grading of Proposition \ref{prop:redModZeroModEight}.
\end{enumerate}
\end{Rem}

The key tool used to prove Theorem \ref{thm:cgpVsZhat}---as in the case for $\sltwo$---is the following reciprocity of quadratic Gauss sums.

\begin{Prop}[{\cite{jeffrey1992,deloup2007}}]
\label{prop:GaussRec}
Keeping the notation of Section \ref{sec:plumbed3Mfld}, there is an equality
\begin{equation*}
\label{eq:GaussRec}
\sum_{n \in \Z^{V(\Gamma)} \slash r \Z^{V(\Gamma)}} e^{\frac{2 \pi \I}{r} (n^t B n + p^t n)}
=
\frac{e^{\frac{\pi \I \sigma}{4}} (\frac{r}{2})^{\frac{\vert V(\Gamma) \vert}{2}}}{\vert \det B \vert^{\frac{1}{2}}} \sum_{\tilde{a} \in \Z^{V(\Gamma)} \slash 2B \Z^{V(\Gamma)}} e^{-\frac{\pi \I r}{2} \left(\tilde{a} + \frac{p}{r} \right)^t B^{-1} \left(\tilde{a} + \frac{p}{r} \right)}.
\end{equation*}
\end{Prop}

\subsection{Factorization when \texorpdfstring{$r \not\equiv 0 \mod 8$}{r ≠ 0 mod 8}}
\label{sec:factrNequiv0}

We begin with the case $r \not\equiv 0 \mod 8$. Using equation \eqref{eq:plumbingCGP}, we see that the CGP invariant factorizes as
\begin{equation}
\label{eq:cgpFactorization}
N_r^{\osp}(M,\coh)=\mathcal{A} \mathcal{B} \mathcal{C},
\end{equation}
where
\[
\mathcal{A}
=
\frac{q^{-\frac{1}{2}(\red{r}-1)^2 \tr B}}{\Delta_+^{b_+} \Delta_-^{b_-}},
\]
\[
\mathcal{B}
=
F^-\left( \{e^{2 \pi \I \frac{\red{r}}{r}\alpha_v} \}_{v \in V(\Gamma)} \right)^{-1}
\]
and
\[
\mathcal{C}
=
\sum_{k \in I_r^{V(\Gamma)}} F^+(\{q^{\alpha_{k_v}}\}_{v \in V(\Gamma)}) q^{\frac{1}{2}(\alpha +k)^t B (\alpha + k)}.
\]
The existence of the factorization \eqref{eq:cgpFactorization} stems from the fact that $\frac{\red{r}}{r} =2$ (resp. $\frac{\red{r}}{r}=1$) when $r \equiv 1 \mod 2$ (resp. $r \equiv 2 \mod 4$), so that the exponents $2 \pi \I \frac{\red{r}}{r}\alpha_v$ appearing in the arguments of $F^-$ are invariant under the shifts in $\alpha_v$ by $2$ which occur in the indexing set $I_r$ of Kirby colours. When $r \equiv 0 \mod 8$---discussed in Section \ref{sec:0Mod8} below---there is invariance only up to a sign, leading to a mild correction in $\mathcal{C}$.

Fix $l \in \Z^{V(\Gamma)}$. Expanding the function $F^+$, the contribution of the monomial $\prod_{v \in V(\Gamma)} x_v^{l_v}$ to $\mathcal{C}$ is
\begin{eqnarray*}
\mathcal{C}_l
&=&
\sum_{k \in I_r^{V(\Gamma)}}
q^{\frac{1}{2}(\alpha +k)^t B (\alpha + k) + l^t (\alpha+k)}.
\end{eqnarray*}
Defining $\tilde{\alpha}=\alpha -(\red{r}-1) \colOne$ and inserting the explicit definitions of $q$ and $I_r$ gives
\[
\mathcal{C}_l
=
e^{\frac{\pi \I}{r}\tilde{\alpha}^t B \tilde{\alpha} + \frac{2 \pi \I}{r} l^t \tilde{\alpha}} \sum_{n \in \Z^{V(\Gamma)} \slash \parity{r} \Z^{V(\Gamma)}} e^{\frac{2 \pi \I}{\parity{r}} n^t (\frac{2}{\even}B) n + \frac{2\pi \I}{\parity{r}}\frac{2}{\even}(l+B\tilde{\alpha})^t n}.
\]
Recall that $\even = \gcd(2,r)$ and $\parity{r} = \frac{r}{\even}$. Applying Proposition \ref{prop:GaussRec}, we obtain
\[
\mathcal{C}_l=e^{- \frac{\pi \I}{r} l^t B^{-1} l}\frac{e^{\frac{\pi \I \sigma}{4}}(\frac{r}{4})^{\frac{\vert V(\Gamma) \vert}{2}}}{\vert \det B \vert^{\frac{1}{2}}} \mathcal{C}^{\prime}_l,
\]
where we have introduced
\[
\mathcal{C}^{\prime}_l
=
\sum_{\tilde{a} \in \Z^{V(\Gamma)} \slash \frac{4}{\even}B \Z^{V(\Gamma)}} e^{-\frac{\pi \I r}{4} \tilde{a}^t B^{-1} \tilde{a} - \pi \I \tilde{a}^t B^{-1}(l + B \tilde{\alpha})}.
\]
Setting $\tilde{a} = BA + a$ for $A \in \Z^{V(\Gamma)} \slash \frac{4}{\even} \Z^{V(\Gamma)}$ and $a \in \Z^{V(\Gamma)} \slash B \Z^{V(\Gamma)}$, we have
\[
\mathcal{C}^{\prime}_l
=
\sum_{\substack{a \in \Z^{V(\Gamma)} \slash B \Z^{V(\Gamma)} \\ A \in \Z^{V(\Gamma)} \slash \frac{4}{\even} \Z^{V(\Gamma)}}} e^{-\frac{\pi \I r}{4} A^t B A- \frac{\pi \I r}{4} a^t B^{-1} a - \frac{\pi \I r}{2} A^t a - \pi \I A^t(l + B \tilde{\alpha}) - \pi \I a^t B^{-1}(l + B \tilde{\alpha})}.
\]
Write $l \equiv 2b + B(s- \colOne) \mod 2 B \Z^{V(\Gamma)}$, where $b \in \Z^{V(\Gamma)} \slash B \Z^{V(\Gamma)}$ and $s \in \left(\Z\slash 2 \Z \right)^{V(\Gamma)}$ satisfies $\sum_{j \in V(\Gamma)} B_{ij} s_j \equiv B_{ii} \mod 2$.
The final two terms in the exponent of a summand of $\mathcal{C}^{\prime}_l$ become
\[
- 2\pi \I A^tb - 2\pi \I a^t B^{-1} b - \pi \I a^t(s- \colOne + \tilde{\alpha}) - \pi \I A^t B(s-\colOne) -\pi \I A^t B \tilde{\alpha}.
\]
Note that $-2\pi \I A^tb \in 2 \pi \I \Z$. Since $\alpha$ is a cohomology class valued in $\Gr = \C \slash 2 \Z$ and $\red{r}$ is even, we have $A^t B \tilde{\alpha} \equiv A^t B \colOne \mod 2 \Z$; see equation \eqref{eq:firstCohom} and Remark \ref{rem:middleWeight}. Note also that $a^t(s- \colOne + \tilde{\alpha}) =a^t(s + \alpha -\red{r} \colOne)$
is congruent to $a^t (s+\alpha)$ modulo $2 \Z$, again because $\red{r}$ is even. We therefore have
\[
\mathcal{C}^{\prime}_l
=
\sum_{a \in \Z^{V(\Gamma)} \slash B \Z^{V(\Gamma)}} e^{- \frac{\pi \I r}{4} a^t B^{-1} a  - 2\pi \I a^t B^{-1} b - \pi \I a^t(s+ \alpha)} \mathcal{C}^{\prime \prime}_l,
\]
where we have set
\[
\mathcal{C}^{\prime \prime}_l
=
\sum_{A \in \Z^{V(\Gamma)} \slash \frac{4}{\even} \Z^{V(\Gamma)}}
e^{-\frac{\pi \I r}{4} A^t B A  - \frac{\pi \I r}{2} A^t a  - \pi \I A^t Bs}.
\]
A more refined case-by-case analysis is now required.

\subsubsection{\texorpdfstring{The case $r \equiv \pm 1 \mod 8$}{The case r = ± 1 mod 8}}
\label{sec:pm1Mod8}

Suppose that $r \equiv \delta \mod 8$ with $\delta = \pm 1$. Since $\even=\gcd(r,2) =1$, we have
\[
\mathcal{C}^{\prime \prime}_l
=
\sum_{A \in \Z^{V(\Gamma)} \slash 4 \Z^{V(\Gamma)}}
e^{\frac{2\pi \I}{4} A^t (-\frac{\delta}{2}B) A + \frac{2\pi \I}{4}(-\delta a-2Bs)^t A}.
\]
Applying Proposition \ref{prop:GaussRec} to $\mathcal{C}^{\prime \prime}_l$, we obtain
\begin{multline*}
\mathcal{C}^{\prime}_l
=
\frac{e^{-\delta \frac{\pi \I \sigma}{4}} 2^{\frac{\vert V(\Gamma) \vert}{2}}}{\vert \det B \vert^{\frac{1}{2}}} \sum_{a, f \in \Z^{V(\Gamma)} \slash B \Z^{V(\Gamma)}} e^{- \frac{\pi \I r}{4} a^t B^{-1} a  - 2\pi \I a^t B^{-1} b - \pi \I a^t\alpha} \cdot \\ e^{4\delta \pi \I f^t B^{-1} f - 2 \pi \I f^t B^{-1} a + \delta  \frac{\pi \I}{4} a^t B^{-1} a + \delta \pi \I s^t B s}.
\end{multline*}
In terms of the formulae of Section \ref{sec:surgeryTopology}, this gives
\begin{multline*}
\mathcal{C}_l
=
e^{-\frac{\pi \I}{r} l^t B^{-1} l} \frac{e^{\frac{\pi \I \sigma (1-\delta)}{4}}r^{\frac{\vert V(\Gamma) \vert}{2}}}{\vert \det B \vert} \cdot \\ \sum_{a,f \in \Z^{V(\Gamma)} \slash B \Z^{V(\Gamma)}} e^{2\pi \I \big(-\frac{r-\delta}{8} \lk(a,a) -\lk(a,b) -\frac{1}{2} \coh(a) + 2 \delta \lk(f,f) - \lk(f,a) +\frac{\delta}{2} s^t B s \big)}.
\end{multline*}

Direct computations give
\[
\mathcal{A}
=
(-1)^{b_+}
r^{-\frac{\vert V(\Gamma) \vert}{2}} q^{\frac{3 \sigma - \tr B}{2}} \cdot
\begin{cases}
e^{\pi \I \sigma} & \mbox{ if } \delta = 1,\\
e^{\frac{\pi \I}{2} \sigma} & \mbox{ if } \delta = -1
\end{cases}
\]
and $\mathcal{B} = (-1)^{b_+} \tor(M,[4 \coh])$; see \cite[Eqn. (A.1)]{costantino2023} for the latter. Combining the above computations, we find
\begin{multline*}
N_r^{\osp}(M,\coh)
=
q^{\frac{3\sigma-\tr B}{2}} e^{\pi \I(\sigma-s^t B s)} \frac{\tor(M,[4 \coh])}{\vert H_1(M;\Z) \vert} \cdot \\
\sum_{\substack{l \in \Z^{V(\Gamma)} \\ a,f \in \Z^{V(\Gamma)} \slash B \Z^{V(\Gamma)}}} F^+_l e^{- \frac{\pi \I}{r} l^t B^{-1} l}  e^{2\pi \I \big(-\frac{r-\delta}{8}\lk(a,a) -\lk(a,b+f) +2 \delta \lk(f,f) - \frac{1}{2} \coh(a)\big)}.
\end{multline*}
Noting that $e^{- \frac{\pi \I}{r} l^t B^{-1} l} = q^{- \frac{1}{2}l^t B^{-1} l}$ and using equation \eqref{eq:RokhlinMod4}
we arrive at the claimed equality.

\subsubsection{\texorpdfstring{The case $r \equiv \pm 2 \mod 8$}{The case r = ± 2 mod 8}}
\label{sec:pm2Mod8}

Suppose now that $r \equiv 2\delta \mod 8$ with $\delta = \pm 1$. Starting from the end of Section \ref{sec:factrNequiv0} and using that $\even =2$, we have
\[
\mathcal{C}^{\prime \prime}_l
=
\sum_{A \in \Z^{V(\Gamma)} \slash 2 \Z^{V(\Gamma)}}
e^{\frac{2\pi \I}{2} A^t \left(-\frac{\delta}{2}B \right) A + \frac{2 \pi \I}{2} (-a-Bs)^tA}.
\]
Applying Proposition \ref{prop:GaussRec}, we obtain
\begin{multline*}
\mathcal{C}^{\prime}_l
=
\frac{e^{-\delta \frac{\pi \I \sigma}{4}} 2^{\frac{\vert V(\Gamma) \vert}{2}}}{\vert \det B \vert^{\frac{1}{2}}} \sum_{a,f \in \Z^{V(\Gamma)} \slash B \Z^{V(\Gamma)}} \\
e^{- \frac{\pi \I(r-2\delta)}{4} a^t B^{-1} a - 2\pi \I a^t B^{-1} b - \pi \I a^t \alpha + 2\delta\pi \I f^t B^{-1} f -2\delta \pi \I f^t B^{-1} a + \frac{\delta\pi \I}{2} s^t B s}.
\end{multline*}
Recombining terms, we conclude that
\begin{multline*}
\mathcal{C}_l
=
e^{-\frac{\pi \I}{r} l^t B^{-1} l}\frac{e^{\frac{\pi \I \sigma}{4} (1-\delta)}(\frac{r}{2})^{\frac{\vert V(\Gamma) \vert}{2}}}{\vert \det B \vert} \cdot \\
\sum_{a,f \in \Z^{V(\Gamma)} \slash B \Z^{V(\Gamma)}} e^{2\pi \I\big(-\frac{r-2 \delta}{8} \lk(a,a) - \lk(a,b+\delta f) - \frac{1}{2} \coh(a) + \delta \lk(f,f) + \frac{\delta}{4} s^t B s \big)}.
\end{multline*}

Direct computations give
\[
\mathcal{A}
=
(-1)^{b_+}
\left(\frac{r}{2}\right)^{-\frac{\vert V(\Gamma) \vert}{2}} q^{\frac{3 \sigma-\tr B}{2}} \cdot
\begin{cases}
e^{-\frac{\pi \I}{2} \sigma} & \mbox{ if } \delta = 1,\\
1 & \mbox{ if } \delta = -1
\end{cases}
\]
and $\mathcal{B} = (-1)^{b_+} \tor(M,[2 \coh])$. Putting the above calculations together gives
\begin{multline*}
N_r^{\osp}(M,\coh)
=
\frac{e^{-\delta \frac{\pi \I}{2} \mu(M,s)} \tor(M,[2 \coh])}{\vert H_1(M;\Z) \vert} \cdot \\ 
\lim_{\qq \rightarrow e^{\frac{4\pi \I}{r}}} \sum_{a,b,f \in H_1(M;\Z)} e^{2 \pi \I \big(-\frac{r-2\delta}{8} \lk(a,a) - \lk(a,b+\delta f) - \delta \lk(f,f) -\frac{1}{2} \coh(a) \big)} \Zhat^{\osp}_{\sigma(b,s)} (M;\qq).
\end{multline*}

\subsubsection{\texorpdfstring{The case $r \equiv \pm 3 \mod 8$}{The case r = ± 3 mod 8}}
\label{sec:pm3Mod8}

While the calculations in the previous two sections can be repeated when $r \equiv 3 \delta \mod 8$, they do not lead to similar conclusions. Indeed, the first step is to write
\[
\mathcal{C}^{\prime \prime}_l
=
\sum_{A \in \Z^{V(\Gamma)} \slash 4 \Z^{V(\Gamma)}}
e^{\frac{2\pi \sqrt{-1}}{4} A^t (-\frac{3 \delta}{2}B) A + \frac{2\pi \sqrt{-1}}{4}(-3 \delta a-2Bs)^t A}.
\]
Unfortunately, due to the coefficient $-\frac{3 \delta}{2}$ of the bilinear form $B$, an application of Proposition \ref{prop:GaussRec} leads to an expression which involves a sum over $\Z^{V(\Gamma)} \slash 3 B \Z^{V(\Gamma)}$, which does not have an obvious topological interpretation in terms of $M$. For this reason, we do not arrive at a universal topological relation between $N_r^{\osp}$ and $\Zhat^{\osp}$.

\subsection{\texorpdfstring{The case $r \equiv 0 \mod 8$}{The case r = 0 mod 8}}
\label{sec:0Mod8}

We require the following result.

\begin{Lem}
\label{lem:FmShift}
Given signs $n \in \{\pm 1\}^{V(\Gamma)}$, there is an equality
\[
F^-(\{n_v x_v\}_{v \in V(\Gamma)}) = (-1)^{\sum_{v \in V(\Gamma)} n_v \deg v} F^-(\{x_v\}_{v \in V(\Gamma)}).
\]
\end{Lem}

\begin{proof}
This is a direct calculation.
\end{proof}

Set $\tilde{\alpha}=\alpha -(\red{r}-1) \colOne$. Since $\frac{\red{r}}{r} = \frac{1}{2}$, Lemma \ref{lem:FmShift} leads to the following modification of the factorization \eqref{eq:cgpFactorization}:
\[
\mathcal{A}
=
(-1)^{b_+} r^{-\frac{\vert V(\Gamma) \vert}{2}} q^{\frac{3 \sigma - \tr B}{2}} e^{\frac{\pi \sqrt{-1}}{4}(3\sigma +4 \tr B)},
\]
\[
\mathcal{B} = (-1)^{b_+} \tor(M,[\coh]),
\]
\[
\mathcal{C}
=
\sum_{n \in (\Z \slash r \Z)^{V(\Gamma)}} (-1)^{\sum_{i \in V(\Gamma)} n_i \deg i} F^+(\{q^{\tilde{\alpha}_{n_j}}\}_{j \in V(\Gamma)}) q^{\frac{1}{2}(\tilde{\alpha} +n)^t B (\tilde{\alpha} + n)}.
\]

Fix $l \in \Z^{V(\Gamma)}$ and write $l = 2b + B(s - \colOne)$ with $\sum_{j \in V(\Gamma)} B_{ij}s_j \equiv B_{ii} \mod 2$. We have
\begin{equation*}
\label{eq:altSumDegree}
\sum_{v \in V(\Gamma)} n_v \deg v \equiv n^t B(s+ \colOne) \mod 2.
\end{equation*}
The contribution of $\prod_{v \in V(\Gamma)} x_v^{l_v}$ to $\mathcal{C}$ is
\begin{eqnarray*}
\mathcal{C}_l
&=&
\sum_{n \in (\Z \slash r \Z)^{V(\Gamma)}}
q^{\frac{1}{2}(\tilde{\alpha}+n)^t B (\tilde{\alpha} + n) + l^t (\tilde{\alpha}+n)+\frac{r}{2}n^t B(s+ \colOne)}.
\end{eqnarray*}
Proceeding as in the previous sections, we arrive at the expression
\begin{multline*}
N_r^{\osp}(M,\coh)
=
q^{\frac{3 \sigma - \tr B}{2}} \frac{e^{\pi \I \mu(M,s)} \tor(M,[\coh])}{\vert H_1(M;\Z) \vert} \sum_{l \in \Z^{V(\Gamma)}} \sum_{a \in \Z^{V(\Gamma)} \slash B \Z^{V(\Gamma)}} F^+_l e^{- \frac{\pi \I}{r} l^t B^{-1} l} \\
e^{-\pi \I r \lk(a,a) - 4 \pi \I \lk(a,b)-2\pi \I \coh(a) - \pi \I \tilde{\alpha}^t B(s+\colOne)}.
\end{multline*}
However, since the final term $e^{- \pi \I \tilde{\alpha}^t B(s+\colOne)}$ is non-topological, we do not obtain a universal relation with $\Zhat^{\osp}$ when $r \equiv 0 \mod 8$.

\subsection*{Conflict of interest statement}
On behalf of all authors, the corresponding author states that there is no conflict of interest.

\bibliographystyle{amsalpha}
\bibliography{ospBib}

\end{document}